\documentclass{amsart}
\usepackage{amsfonts,amsmath,amsthm,amssymb,color}
\usepackage[all]{xy}
\usepackage[colorlinks=true, linkcolor=blue, citecolor=blue]{hyperref} 
\usepackage{faktor}
\usepackage{graphicx}

\DeclareMathOperator*{\colim}{colim}
\numberwithin{equation}{section}

\begin{document}

\newtheorem{theorem}{Theorem}[section]
\newtheorem{axiom}[theorem]{Axiom}
\newtheorem{lemma}[theorem]{Lemma}
\newtheorem{proposition}[theorem]{Proposition}
\newtheorem{corollary}[theorem]{Corollary}

\theoremstyle{definition}
\newtheorem{definition}[theorem]{Definition}
\newtheorem{example}[theorem]{Example}

\theoremstyle{remark}
\newtheorem{remark}[theorem]{Remark}

\newenvironment{magarray}[1]
{\renewcommand\arraystretch{#1}}
{\renewcommand\arraystretch{1}}

\newenvironment{acknowledgement}{\par\addvspace{17pt}\small\rm
\trivlist\item[\hskip\labelsep{\it Acknowledgement.}]}
{\endtrivlist\addvspace{6pt}}

\newcommand{\quot}[2]{
{\lower-.2ex \hbox{$#1$}}{\kern -0.2ex /}
{\kern -0.5ex \lower.6ex\hbox{$#2$}}}

\newcommand{\mapor}[1]{\smash{\mathop{\longrightarrow}\limits^{#1}}}
\newcommand{\mapin}[1]{\smash{\mathop{\hookrightarrow}\limits^{#1}}}
\newcommand{\mapver}[1]{\Big\downarrow
\rlap{$\vcenter{\hbox{$\scriptstyle#1$}}$}}
\newcommand{\liminv}{\smash{\mathop{\lim}\limits_{\leftarrow}\,}}

\newcommand{\sSet}{\mathbf{sSet}}
\newcommand{\Set}{\mathbf{Set}}
\newcommand{\Art}{\mathbf{Art}}
\newcommand{\CDGA}{\mathbf{CDGA}}
\newcommand{\CGA}{\mathbf{CGA}}
\newcommand{\DGCA}{\mathbf{CDGA}}
\newcommand{\sCA}{\mathbf{sCA}}
\newcommand{\Top}{\mathbf{Top}}
\newcommand{\DGMod}{\mathbf{DGMod}}
\newcommand{\DGLA}{\mathbf{DGLA}}
\newcommand{\DGArt}{\mathbf{DGArt}}
\newcommand{\bM}{\mathbf{M}}
\newcommand{\bN}{\mathbf{N}}
\newcommand{\bC}{\mathbf{C}}

\newcommand{\solose}{\Rightarrow}
\newcommand{\PSI}{\Psi\mathbf{Sch}_I(\mathbf{M})}
\newcommand{\PSJ}{\Psi\mathbf{Sch}_J(\mathbf{M})}
\newcommand{\Sym}{\mbox{Sym}}

\newcommand{\specif}[2]{\left\{#1\,\left|\, #2\right. \,\right\}}

\renewcommand{\bar}{\overline}
\newcommand{\de}{\partial}
\newcommand{\debar}{{\overline{\partial}}}
\newcommand{\per}{\!\cdot\!}
\newcommand{\Oh}{\mathcal{O}}
\newcommand{\sA}{\mathcal{A}}
\newcommand{\sB}{\mathcal{B}}
\newcommand{\sC}{\mathcal{C}}
\newcommand{\sE}{\mathcal{E}}
\newcommand{\sF}{\mathcal{F}}
\newcommand{\sG}{\mathcal{G}}
\newcommand{\sH}{\mathcal{H}}
\newcommand{\sI}{\mathcal{I}}
\newcommand{\sJ}{\mathcal{J}}
\newcommand{\sK}{\mathcal{K}}
\newcommand{\sL}{\mathcal{L}}
\newcommand{\sM}{\mathcal{M}}
\newcommand{\sN}{\mathcal{N}}
\newcommand{\sP}{\mathcal{P}}
\newcommand{\sU}{\mathcal{U}}
\newcommand{\sV}{\mathcal{V}}
\newcommand{\sX}{\mathcal{X}}
\newcommand{\sY}{\mathcal{Y}}
\newcommand{\sW}{\mathcal{W}}
\newcommand{\Ni}{\mathcal{N}_I}
\newcommand{\Nj}{\mathcal{N}_J}
\newcommand{\PM}{\Psi\mathbf{Mod}}
\newcommand{\QCoh}{\mathbf{QCoh}}
\newcommand{\DGSch}{\mathbf{DGSch}}
\newcommand{\DGAff}{\mathbf{DGAff}}

\newcommand{\Aut}{\operatorname{Aut}}
\newcommand{\Cof}{\operatorname{Cof}}
\newcommand{\Fun}{\operatorname{Fun}}
\newcommand{\Imm}{\operatorname{Imm}}
\newcommand{\Mor}{\operatorname{Mor}}
\newcommand{\Map}{\operatorname{Map}}
\newcommand{\Def}{\operatorname{Def}}
\newcommand{\D}{\operatorname{D}}
\newcommand{\Ho}{\operatorname{Ho}}
\newcommand{\Mod}{\operatorname{Mod}}
\newcommand{\Hom}{\operatorname{Hom}}
\newcommand{\Hilb}{\operatorname{Hilb}}
\newcommand{\HOM}{\operatorname{\mathcal H}\!\!om}
\newcommand{\DER}{\operatorname{\mathcal D}\!er}
\newcommand{\Spec}{\operatorname{Spec}}
\newcommand{\Der}{\operatorname{Der}}
\newcommand{\Tor}{{\operatorname{Tor}}}
\newcommand{\Ext}{{\operatorname{Ext}}}
\newcommand{\End}{{\operatorname{End}}}
\newcommand{\END}{\operatorname{\mathcal E}\!\!nd}
\newcommand{\Image}{\operatorname{Im}}
\newcommand{\coker}{\operatorname{coker}}
\newcommand{\tot}{\operatorname{tot}}
\newcommand{\Id}{\operatorname{Id}}
\newcommand{\cone}{\operatorname{cone}}
\newcommand{\cocone}{\operatorname{cocone}}
\newcommand{\cyl}{\operatorname{cyl}}
\newcommand{\id}{\operatorname{id}}
\newcommand{\mA}{\mathfrak{m}_{A}}

\renewcommand{\Hat}[1]{\widehat{#1}}
\newcommand{\dual}{^{\vee}}
\newcommand{\desude}[2]{\dfrac{\de #1}{\de #2}}

\newcommand{\A}{\mathbb{A}}
\newcommand{\SExt}{\mathbb{S}\mbox{Ext}}
\newcommand{\N}{\mathbb{N}}
\newcommand{\R}{\mathbb{R}}
\newcommand{\Z}{\mathbb{Z}}
\renewcommand{\H}{\mathbb{H}}
\renewcommand{\P}{\mathbb{P}}
\renewcommand{\L}{\mathbb{L}}
\newcommand{\proj}{\mathbb{P}}
\newcommand{\Proj}{\mathbb{P}}
\newcommand{\K}{\mathbb{K}\,}
\newcommand\C{\mathbb{C}}
\newcommand\Del{\operatorname{Del}}
\newcommand\Tot{\operatorname{Tot}}
\newcommand\Grpd{\mbox{\bf Grpd}}

\newcommand\é{\'e}
\newcommand\è{\`e}
\newcommand\à{\`a}
\newcommand\ì{\`i}
\newcommand\ù{\`u}
\newcommand\ò{\`o }


\newcommand{\rh}{\rightarrow}
\newcommand{\contr}{{\mspace{1mu}\lrcorner\mspace{1.5mu}}}

\newcommand{\bi}{\boldsymbol{i}}
\newcommand{\pal}{{\boldsymbol{\cdot}}}
\newcommand{\pallino}{{\scriptscriptstyle\bullet}}

\newcommand{\bl}{\boldsymbol{l}}

\newcommand{\MC}{\operatorname{MC}}
\newcommand{\TW}{\operatorname{TW}}

\title[Deformations of schemes via Reedy-Palamodov resolutions]{Deformations of algebraic schemes via Reedy-Palamodov cofibrant resolutions}


\author{Marco Manetti}
\address{\newline
Universit\`a degli studi di Roma La Sapienza,\hfill\newline
Dipartimento di Matematica \lq\lq Guido
Castelnuovo\rq\rq,\hfill\newline
P.le Aldo Moro 5,
I-00185 Roma, Italy.}
\email{manetti@mat.uniroma1.it}
\urladdr{www.mat.uniroma1.it/people/manetti/}

\author{Francesco Meazzini}
\email{meazzini@mat.uniroma1.it}
\urladdr{www.mat.uniroma1.it/people/meazzini/}

\subjclass[2010]{18G55,14D15,16W50}
\keywords{Model categories, Deformation theory, Differential graded algebras, Algebraic schemes, Cotangent complex}

\begin{abstract} Let $X$ be a Noetherian separated and finite dimensional scheme over a field $\K$ of characteristic zero. The goal of this paper is to study deformations of $X$ over a differential graded local Artin $\K$-algebra by using local Tate-Quillen resolutions, i.e., the algebraic analog of the Palamodov's resolvent of a complex space.
The above goal is achieved by describing the DG-Lie algebra controlling deformation theory of a diagram of differential graded commutative algebras, indexed by a direct Reedy category.  
\end{abstract}

\maketitle

\tableofcontents

\bigskip

\section{Introduction}

This paper concerns the  use of basic model category theory in the study of deformations of algebraic schemes and morphisms between them, with the aim of being accessible to a wide community,  especially to everyone having a classical background in algebraic geometry and deformation theory. For this reason 
the homotopic and simplicial background is reduced at minimum.

Let $X$ be a Noetherian separated and finite dimensional scheme over a field $\K$ of characteristic zero;  we study deformations of $X$ over a differential graded local Artin $\K$-algebra by using local Tate-Quillen resolutions, i.e., the algebraic analog of the Palamodov's resolvent of a complex space.

It is well known (see e.g. \cite{Sernesi}) that if $X=\Spec(S)$ is affine, then the deformations of $X$ are the same as the deformations of $S$ in the category of commutative algebras. The latter are studied by using a Tate-Quillen resolution $R\to S$, and are controlled by the DG-Lie algebra $\Der^*_{\K}(R,R)$ via Maurer-Cartan equation modulus gauge action.  

For general schemes,   fixing an affine open cover $\{U_i\}_{i\in I}$, the geometry of  $X$ is encoded in the diagram
\[ S_\pal\colon \sN\to \K\text{-algebras} \; , \qquad \qquad \alpha\mapsto S_{\alpha}=\Gamma(U_{\alpha},\Oh_X) \]
where $\sN$ denotes the nerve of the cover; the deformation theory of $X$ is equivalent to the one of $S_{\pal}$. 

Since Tate-Quillen resolutions are cofibrant objects in the model category $\CDGA_{\K}^{\le 0}$ of commutative differential (non-positively) graded algebras, it is natural to consider $S_\pal$ as an element of the category  
$\Fun(\sN,\CDGA_{\K}^{\le 0})$ of functors $\sN\to \CDGA_{\K}^{\le 0}$: 
\begin{equation}\label{eq.diagramX}
S_{\pal}\colon \sN\to \CDGA_{\K}^{\le 0} \; , \qquad \qquad \alpha\mapsto S_{\alpha}=\Gamma(U_{\alpha},\Oh_X)
\end{equation}
where each $S_{\alpha}$ is considered as a DG-algebra concentrated in degree $0$.

Since $\sN$ is a (direct) Reedy category then $\Fun(\sN,\CDGA_{\K}^{\le 0})$ is endowed with the Reedy model structure 
(see Section~\ref{sec.diagramoverReedy}) that is strong left proper (Proposition~\ref{prop.flatnessReedy}), in the sense of \cite{MM}, briefly recalled here in Definition~\ref{def.slp}.  According to the results of 
\cite{MM} there exists a good  deformation theory in  strong left proper model categories 
 that, among the other properties, is homotopy invariant: 
 in our particular case the deformation theory  of any diagram $R_\pal$ gives a ``deformation'' functor 
 \[ \Def_{R_\pal}\colon \DGArt_{\K}^{\le 0}\to \Set,\]
and for any diagram $P_\pal$ weak equivalent to $R_\pal$ we have an isomorphism of functors $\Def_{P_\pal}\simeq \Def_{R_\pal}$.

It is easy to prove (Lemma~\ref{lem.D=Def_classic}) that the restriction of $\Def_{S_\pal}$ to the subcategory of local Artin algebras concentrated in degree 0 is the same as the classical deformation functor of $X$. Therefore 
the above facts provide a natural way to define deformations of $X$ over general DG-Artin local ring in non-positive degrees; moreover we can replace the diagram $S_\pal$ with any weak equivalent Reedy cofibrant diagram. It is worth to notice that 
the algebraic analog of the Palamodov's resolvent \cite{Pala1,Pala2} is in fact a special case of Reedy cofibrant replacement.

Finally, for any Reedy cofibrant diagram $R_\pal$, we shall be able to prove (Lemma~\ref{lem.strict=large_cofib} and Theorem~\ref{thm.DefCofibrantDiagram}) that 
the functor $\Def_{R_\pal}$ is controlled by the DG-Lie algebra of derivations of $R_\pal$. 

The proposed proof strongly relies on the results of \cite{MM}, where deformations of  affine schemes were considered.
More precisely, we show that the ideas developed in~\cite{MM} in order to understand the deformation theory of an affine (differential graded) scheme can be easily adapted to the non-affine case. Philosophically, we can say that  the approach to deformation theory via model categories presented in this paper and in \cite{MM}, gives not only similar statements in the affine and non-affine case, but also the same  underlying ideas and strategies in the proofs.

The same approach leads to the description of the cotangent complex as a certain homotopy class of $S_\pal$-modules, namely the module of K\"{a}hler differentials of a cofibrant replacement (see Section~\ref{section.cotangentcomplex}). 
This description relies on the results of \cite{FM}, where it is proved that the homotopy category of 
$S_\pal$-modules is equivalent to the unbounded derived category of quasi-coherent sheaves on $X$.

As also suggested by the referee, it would be interesting, instead of just looking at functors of homotopy classes, to consider simplicial functors as in \cite{hinichDGC} with moduli interpretations as infinity groupoids as in the Derived Algebraic Geometry literature (Lurie, Pridham, To\"{e}n, Vezzosi etc.): 
the results of this  paper easily extend from schemes to derived schemes. However, in view of the general philosophy underlying this paper, we preferred to consider this possible extension, possibly after the simplicial analogous of the general formal theory developed in \cite{MM}.

\bigskip

\section{Deformations of  morphisms in strong left proper model categories}\label{sec.recalldefomodel}

The goal of this section is to fix notation and to review some results of~\cite{MM}.

Let $\bM$ be a fixed model category. For every object $A\in \bM$ the symbols $A\downarrow \bM$ and 
$\bM_A$ both denote the undercategory of maps $A\to B$, $B\in \bM$. 
There is a natural model structure on $\bM_A$ under which a map is a cofibration, fibration, or weak equivalence if
and only if its image in $\bM$ under the forgetful functor is,
respectively, a cofibration, fibration, or weak equivalence. 
Every 
morphism $A\to B$ induces a push-out functor $-\amalg_A B\colon \bM_A\to \bM_B$ which preserves 
cofibrations and trivial cofibrations.

\begin{definition}[Definitions 2.9 and 2.13 of \cite{MM}]\label{def.slp}
A morphism $A\to B$ in a model category $\bM$ is called \emph{flat} if the push-out functor
$-\amalg_A B\colon \bM_A\to \bM_B$ preserves pull-back squares of trivial fibrations.
The model category $\bM$ is said to be \emph{strong left proper} if every cofibration is flat.
\end{definition}

Thus, in a strong left proper model category, the push-out along a cofibration preserves trivial cofibrations and trivial fibrations; hence preserves weak equivalences, i.e. the model category is left proper. Conversely, a left proper model category may not be strong:
for instance, the category of topological spaces endowed with the usual model structure is left proper but not strong left proper.

We refer to \cite{MM} for a deeper discussion about flat morphisms and for the proof that the class of flat morphisms is closed under composition, push-outs and retracts.
An object $X$ in $\bM$ is called flat if the morphism from the initial object to $X$ is flat; clearly 
a morphism $A\to M$ is flat as a morphism in $\bM$ if and only if it is flat as an object in the undercategory $A\downarrow\bM$.

According to \cite[Cor. 3.4]{MM} an example of strong left proper model category is  $\CDGA_{\K}^{\le 0}$, the category of differential graded commutative algebras  over a field $\K$ of characteristic 0 concentrated in non-positive degrees,
equipped with the projective model structure (\cite{BG76}, \cite[V.3]{GelfandManin}):  
weak equivalences are the quasi-isomorphisms,   cofibrations are the retracts of semifree extensions and  
fibrations are the surjections in strictly negative degrees. 

It is easy to prove that in a  left proper model category, weak equivalences between flat objects 
are preserved under arbitrary push-outs \cite[2.5+2.11]{MM}. The converse is generally false and this motivates the following definition.

\begin{definition}[Definition 4.2 of \cite{MM}]\label{def.M(K)}
Let $\bM$ be a  left proper model category. A morphism $A\to K$ is said to be  a \emph{thickening} 
if for every commutative diagram 
\begin{equation}\label{equ.thickening} \xymatrix{	A \ar@/_1pc/[dr]^{f} \ar@{->}[r]^{g} & E\ar[d]^h \\
 & D	} \end{equation}
such that $f,g$ are flat and  $h\amalg \Id_K\colon E\amalg_AK\to D\amalg_AK$ is a weak equivalence (respectively: an isomorphism), then also 
 $h$ is a weak equivalence (respectively: an isomorphism).
\end{definition}

For instance, in the model category $\CDGA_{\K}^{\le0}$ every surjective morphism with nilpotent kernel is a thickening \cite[Prop. 3.5]{MM}: the name thickening is clearly motivated by the analogous notion for algebraic schemes 
\cite[8.1.3]{FGAexplained}.

\begin{definition}\label{def.deformationXL} 
Let $K\xrightarrow{f}X$ be a morphism in a left-proper model category $\bM$, with $X$ a fibrant object.   
A deformation of $f$ over  a thickening  $A\xrightarrow{p}K$ is a commutative diagram 
\[ \xymatrix{	A\ar[d]^p\ar[r]^{f_A} & X_A\ar[d] \\
K\ar[r]^f & X	} \]
such that $f_A$ is flat and the induced map $X_A\amalg_AK\to X$ is a weak equivalence. 
A direct equivalence is given by a commutative diagram 
\[ \xymatrix{	A\ar[d]_{g_A}\ar[r]^{f_A} & X_A\ar[d]\\
Y_A\ar[r]\ar[ru]^h & X	} \]
where $h$ is a weak equivalence. Two deformations are equivalent if they are so under the equivalence relation generated by direct equivalences.
\end{definition}

We denote either by $\Def_f(A\xrightarrow{p}K)$ or, with a little abuse of notation, by $\Def_f(A)$ the quotient class of deformations of $f$ up to equivalence. Given any diagram $A\xrightarrow{h} B\xrightarrow{p}K$ 
with $p,ph$ thickening, then the push-out along $h$ gives  a natural map 
$h\colon \Def_f(A\xrightarrow{ph}K)\to \Def_f(B\xrightarrow{p}K)$:  every diagram 
\[ \xymatrix{	A\ar[d]^{ph}\ar[r]^{f_A} & X_A\ar[d] \\
K\ar[r]^f & X	} \]
as in Definition~\ref{def.deformationXL} is mapped into the diagram 
\[ \begin{matrix}\xymatrix{	B\ar[d]^p\ar[r]^-{h_*(f_A)} & B\amalg_AX_A\ar[d] \\
K\ar[r]^f & X}\end{matrix}. \]

In strong left proper model categories it is possible to describe the class of deformations exclusively in terms of cofibrations.

\begin{proposition}[Lemma 4.4 and Prop. 4.6 of \cite{MM}]\label{prop.pushoutcfdefo}
Let $K\xrightarrow{f}X$ be a morphism in a strong left proper model category $\bM$, with $X$ a fibrant object.   
Then every deformation of $f$ over  a thickening  $A\xrightarrow{p}K$ is represented by a commutative diagram 
\begin{equation}\label{equ.cdefo}  
\xymatrix{	A\ar[d]^p\ar[r]^{f_A} & X_A\ar[d]^{\alpha} \\
K\ar[r]^f & X	}
\end{equation}
such that $f_A$ is a cofibration and the induced map $X_A\amalg_AK\to X$ is a weak equivalence. 

The diagrams \eqref{equ.cdefo} and 
\begin{equation}\label{equ.cdefobis}  
\xymatrix{	A\ar[d]^p\ar[r]^{g_A} & Y_A\ar[d]^{\beta} \\
K\ar[r]^f & X	}
\end{equation}
with $g_A$ a cofibration, represent the same equivalence class of deformations of $f$ if and only if 
there exists a  commutative diagram 
\begin{equation}\label{equ.cbarequivalence} 
\begin{matrix}\xymatrix{	 & A\ar[dl]_{f_A}\ar[d]\ar[rd]^{g_A} & \\
X_A\ar[r]\ar[dr]_{\alpha} & Z_A\ar[d] & Y_A\ar[l]\ar[dl]^{\beta} \\
 & X & 	}\end{matrix}
\end{equation}
with the horizontal arrows trivial cofibrations.
\end{proposition}

The assumption that $p$ is a thickening is essential for the validity of the following theorem.

\begin{theorem}[Homotopy invariance of deformations: Thm. 5.3 of \cite{MM}]\label{thm.homotopyequivalence}
Let $K\xrightarrow{f}X\xrightarrow{\tau}Y$ be morphisms in a strong left proper model category $\bM$. 
If $\tau$ is a weak equivalence between fibrant objects, then  for every thickening $A\to K$, the natural  map 
\[ \Def_f(A)\to \Def_{\tau f}(A),\qquad (A\to X_A\to X)\mapsto (A\to X_A\to X\xrightarrow{\tau}Y),\] 
is bijective. 
\end{theorem}

Theorem~\ref{thm.homotopyequivalence} implies that it is properly defined the deformation theory of any morphism $f\colon K\to Y$ by setting $\Def_f=\Def_{\tau f}$, where $\tau\colon Y\to X$ is any 
weak equivalence into a fibrant object $X$. At the same time,   
Theorem~\ref{thm.homotopyequivalence} implies that deformation theory (along a thickening) is invariant under fibrant-cofibrant replacements of $f$  in the undercategory $\bM_K$:  for every diagram  
\[\xymatrix{	K\ar[d]^f\ar[r]^{i} & R\ar[d]^{\beta} \\
Y\ar[r]^\alpha & X	}\]
with $X$ fibrant, $i$ cofibration, $\beta$ fibration and $\alpha,\beta$ weak equivalences,  the morphisms $f$ and $i$ have the same deformation theory.

\bigskip
\section{Diagrams over  direct Reedy categories}\label{sec.diagramoverReedy}

Let $\sC$ be a (non empty) direct Reedy category. This means  that $\sC$ is a small category and 
there exists a degree function $Ob(\sC)\to \N$ such that every non-identity morphism raises degree. 
In particular,  every object $a$ has only the identity as a morphism $a\to a$.

Examples of direct Reedy categories are:

\begin{enumerate}
\item the category $\overrightarrow{\Delta}$ of finite ordinals  with injective strictly monotone maps.

\item the category associated to a Reedy poset: by definition a Reedy poset  is a partially ordered set $I$ such that  there exists a strictly monotone map  $\deg\colon I\to \N$, i.e. $\deg(\alpha)<\deg(\beta)$ whenever $\alpha<\beta$. 

\item every finite product of direct Reedy categories is a direct Reedy category, equipped with the degree function $\deg(a_1,\ldots,a_n)=\sum \deg(a_i)$.

\end{enumerate}

From non on we shall denote by $\sC$ a fixed direct Reedy category.  
As usual we shall denote by $\Map(\sC)$ the category of maps in $\sC$: objects are the morphisms in $\sC$, morphisms are the commutative squares. We are mainly interested in the following  full subcategories of $\Map(\sC)$:

\begin{enumerate}

\item for every $a\in \sC$ denote by $[\sC,a]$ the full subcategory of 
$\Map(\sC)$ whose objects are the morphisms $b\to a$. This is naturally isomorphic to the overcategory 
$\sC\downarrow a$.

\item for every $a\in \sC$ denote by $[\sC,a)$ the full subcategory of 
$\Map(\sC)$ whose objects are the non-identity morphisms $b\to a$. 
This is naturally isomorphic to the latching category  
$\de(\sC\downarrow a)$ defined in \cite{Hir03}.

\end{enumerate}

Notice that both $[\sC,a]$ and $[\sC,a)$ are direct Reedy categories, with degree function
$\deg(b\to a)=\deg(b)$.

\bigskip

Let $\bM$ be a fixed model category. 
For every diagram 
$X\colon \sC\to \bM$ and every $a\in \sC$ we may consider  the diagram
\[ \sL_aX\colon [\sC,a)\to \bM,\qquad (b\to a)\mapsto X_b\,.\]
The latching object of $X$ at $a$ is defined as the colimit of the diagram $\sL_aX$:
\[ L_aX=\colim_{[\sC,a)}\sL_aX\,,\]
and the latching map of $X$ at $a$ is the natural map 
$L_aX\to X_a$ 
induced by the natural maps  $(\sL_aX)_{b\to a}=X_b\to X_a$.

The Reedy model structure on  the category $\bM^{\sC}$ of diagrams $X\colon \sC\to \bM$, also denoted by $\Fun(\sC,\bM)$, is defined as follows:

\begin{enumerate} 

\item a morphism $X\to Y$ is a weak-equivalence (respectively, fibration) if for every $a\in \sC$ the morphism $X_a\to Y_a$ is a weak equivalence (respectively, fibration).

\item a morphism $X\to Y$ is a cofibration if for every $a\in \sC$ the natural morphism
\[ X_a\amalg_{L_aX}L_aY\to Y_a\]
is a cofibration.
\end{enumerate}

It is useful to recall that if $X\to Y$ is a Reedy cofibration in $\bM^{\sC}$, then 
$X_a\to Y_a$ and $L_aX\to L_aY$ are  cofibrations in $\bM$ for every $a\in \sC$, \cite[Prop. 15.3.11]{Hir03}. If $\bM$ is left proper, then also $\bM^{\sC}$ is left proper, \cite[Thm. 15.3.4]{Hir03}.

It is important to point out that  Reedy model structures commute with undercategories and overcategories in the following sense: denoting by $\Delta\colon \bM\to \bM^{\sC}$ the diagonal functor, for any object $A\in \bM$ there exist canonical isomorphisms of model categories
\[ \Delta A \downarrow \bM^{\sC}=(A\downarrow\bM)^{\sC},\qquad 
\bM^{\sC}\downarrow \Delta A=(\bM\downarrow A)^{\sC}\,.\]
This is completely trivial since the above natural isomorphisms of categories preserve weak equivalences and fibrations.

Since our goal is to make deformation theory in $\bM^{\sC}$ we  need to characterise  the flat morphisms.

\begin{proposition}\label{prop.flatnessReedy} 
In the above setup, a morphism $X\to Y$ in $\bM^{\sC}$ is flat if 
$X_a\to Y_a$ is flat in $\bM$ for every $a\in \sC$. If $\bM$ is strong left proper, then also 
$\bM^{\sC}$ is strong left proper.   
\end{proposition} 

\begin{proof} Let $X\to Y$ be a morphism in $\bM^{\sC}$, since push-outs and pull-backs are made objectwise, and trivial fibrations are detected objectwise, it is clear that if every $X_a\to Y_a$ is flat, then also $X\to Y$ is flat. 

If $\bM$ is strong left proper and  
$X\to Y$  is a cofibration in $\bM^{\sC}$,  we have seen that $X_a\to Y_a$ is a cofibration for every 
$a\in \sC$, hence $X_a\to Y_a$ is flat for every $a$ and therefore also $X\to Y$ is flat.
\end{proof}

\begin{lemma}\label{lem.flatforcones} In the above setup, a cone  $\Delta A\to Y$ in  $\bM^{\sC}$ is flat if and only if $A\to Y_a$ is flat in $\bM$ for every $a\in \sC$.
\end{lemma}

\begin{proof} One implication is proved in Proposition~\ref{prop.flatnessReedy}. 
The converse is an easy consequence of the fact that the diagonal functor 
$\Delta\colon \bM\to \bM^{\sC}$ preserves pull-back squares of trivial fibrations and 
pull-back squares in $\bM^{\sC}$ are detected objectwise.
\end{proof}

\begin{corollary}\label{cor.Reedythick} 
In the above setup, a morphism 
$A\to K$ in $\bM$ is a thickening if and only if $\Delta A\to \Delta K$ is a thickening in $\bM^{\sC}$.
\end{corollary}
 
\begin{proof} Since the diagonal functor $\Delta$ commutes with push-outs, its application to the diagram 
\eqref{equ.thickening} immediately implies that if $\Delta A\to \Delta K$ is a thickening then 
$A\to K$ is a thickening.  

If  $\Delta A\to Y$ is flat, then 
for every 
$a\in \sC$ we have 
\[\left(\Delta K\amalg_{\Delta A}Y\right)_a=K\amalg_{A}Y_a\]
and the morphism $A\to Y_a$ is flat by Lemma~\ref{lem.flatforcones}: this implies that 
that if $A\to K$ is a thickening then 
$\Delta A\to \Delta K$ is a thickening.  
\end{proof}

Thus, according to Proposition~\ref{prop.flatnessReedy} and the results of Section~\ref{sec.recalldefomodel}, there exists a good deformation theory of  diagrams in a strong left proper model category $\bM$ over a direct Reedy index category $\sC$. 

If we restrict to diagonal thickenings in $\bM^{\sC}$, i.e., to thickenings of the form 
$\Delta A\to \Delta K$ with $A\to K$ a thickening in
$\bM$, by Corollary~\ref{cor.Reedythick} we obtain the following equivalent description of deformations.

\begin{definition} 
Given a thickening $p\colon A\to K$  in $\bM$ and a fibrant diagram
\[X\in (\bM_K)^{\sC}=(K\downarrow \bM)^{\sC}=\Delta K\downarrow \bM^{\sC},\]
a deformation of $X$ along $p$  is a commutative square
\[ \xymatrix{	\Delta A\ar[d]^p\ar[r]^{f} & \sX\ar[d] \\
\Delta K\ar[r]& X	} \]
such for every $a\in \sC$ the map $f_a\colon A\to \sX_a$ is flat and the  map $\sX_a\amalg_AK\to X_a$ is a weak equivalence.
\end{definition}

The main goal of this paper is to study deformations of a diagram with values in the strong left proper 
model category $\CDGA^{\le0}_{\K}$ over an element in the full subcategory 
$\DGArt_{\K}^{\le 0}$ of 
local Artinian DG-algebras with residue field $\K$. This make sense since, according to \cite[Prop. 3.5]{MM} every  surjective morphism in $\DGArt_{\K}^{\le 0}$ is a thickening in the model category $\CDGA^{\le0}_{\K}$, and every diagram in $\Fun(\sC,\CDGA^{\le0}_{\K})$ is Reedy fibrant:
for a direct Reedy category $\sC$, 
a deformation of a diagram $X\in \Fun(\sC,\CDGA^{\le0}_{\K})$ over 
$A\in \DGArt_{\K}^{\le 0}$ is a \emph{flat diagram} 
$X_A\in \Fun(\sC,\CDGA^{\le0}_{A})$ equipped with a weak equivalence 
$X_A\amalg_A\K\to X$, where $X_A \amalg_A \K$ denotes the diagram defined by $(X_A\amalg_A\K)_a=X_{A,a}\amalg_A\K$ for every $a\in\sC$.

\bigskip

\section{Lifting of trivial idempotents}\label{section.trivialidempotents}

By definition, a trivial idempotent in a model category is an endomorphism $e\colon X\to X$  which is a weak equivalence satisfying $e^2=e$. 
The next goal is to prove a lifting result for trivial idempotents that will be essential for the computation of the DG-Lie algebra controlling deformations of diagrams of algebras over direct Reedy categories. We first need a preliminary lemma.

\begin{lemma}\label{lem.latchingtrivialidempotent} 
Let $\bM$ be a left proper model category and $\sC$ a direct Reedy category.
Assume it is given a Reedy cofibration $i\colon P\to R$, an element $a\in \sC$ and a 
morphism of diagrams 
$e\colon \sL_aR\to \sL_aR$ that is a trivial idempotent satisfying $ei=i$. Then  
\[  e\colon P_a\amalg_{L_aP}L_aR\to P_a\amalg_{L_aP}L_aR,\]
is a trivial idempotent in the model category $\bM$. 
\end{lemma}

\begin{proof}
For every diagram $X$ in $\bM^{\sC}$ and every $a\in \sC$ we may write 
$L_aX=\colim \sL_aX$, where 
\[ \sL_aX\colon [\sC,a)\to \bM,\qquad (\sL_aX)_{b\to a}=X_b\,.\]

For every  morphism $f\colon b\to a$ in $[\sC,a)$ there exists a natural bijection
\[[\sC,b)\xrightarrow{\;\simeq\;}[[\sC,a),f)\,,\qquad \xymatrix{c\ar[r]^g&b}\quad\mapsto \quad
\begin{matrix}\xymatrix{c\ar[r]^g\ar[dr]_{fg}&b\ar[d]^f\\ &a}\end{matrix}\,,\]  
and this implies  that the latching functor 
\[ \sL_a\colon \bM^{\sC}\to \bM^{[\sC,a)},\qquad X\mapsto \sL_aX,\]
preserves fibrations, cofibrations and weak equivalences.

In particular $\sL_aP\to\sL_aR$ is a Reedy cofibration, and taking its push-out along the 
natural map $\sL_aP\to \Delta P_a$ we get  a Reedy cofibration  
\[ \Delta P_a\to \Delta P_a\amalg_{\sL_aP}\sL_aR\]
in $\bM^{[\sC,a)}$: equivalently the diagram  
\[ Q\colon [\sC,a)\to \bM_{P_a},\qquad Q_{b\to a}=P_a\amalg_{P_b}R_b,\]
is Reedy cofibrant in  $\bM_{P_a}^{[\sC,a)}$.

For every $b\to a$, since $P_b\to R_b$ is a cofibration and  $\bM$ is left proper, by gluing lemma the idempotent $e\colon P_a\amalg_{P_b}R_b\to P_a\amalg_{P_b}R_b$ is a weak equivalence. 
Since $[\sC,a)$ has fibrant constants, the colimit functor $\colim \colon\bM_{P_a}^{[\sC,a)}\to \bM_{P_a}$ preserves weak equivalences between cofibrant objects and the conclusion follows from the natural isomorphism
\[ \colim_{[\sC,a)}Q=P_a\amalg_{L_aP}L_aR \]
that holds since colimits commute with push-outs.
\end{proof}


We are now ready to use Lemma~\ref{lem.latchingtrivialidempotent} together \cite[Theorem 6.12]{MM} in order to prove the main result of this section. 
For every morphism $A\to B$ in $\CDGA^{\le0}_{A}$ and every direct Reedy category $\sC$ we shall denote by $-\otimes_AB\colon \Fun(\sC,\CDGA^{\le0}_{A})\to \Fun(\sC,\CDGA^{\le0}_{B})$ the natural functor induced by composition with the usual push-out map
$-\otimes_AB\colon \CDGA^{\le0}_{A}\to \CDGA^{\le0}_{B}$.

\begin{theorem}[Lifting of trivial idempotents]\label{thm.lifttrivialidemp} 
Let $\sC$ be a direct Reedy category,  $A\to B$ a surjective morphism 
in $\DGArt_{\K}^{\le 0}$ and  $i\colon X\to Y$ a cofibration of flat diagrams in 
$\Fun(\sC,\CDGA^{\le0}_{A})$. 

Then every trivial idempotent $\epsilon$ of $Y\otimes_AB$, commuting with 
the cofibration $X\otimes_AB\to  Y\otimes_AB$, lifts to a trivial idempotent 
$e_A\colon Y\to Y$ such that $e_Ai=i$.
\end{theorem}

\begin{proof} Define an ideal of $\sC$ as a full subcategory $\sB$ such that if $b\in \sB$ and 
$a\to b$ is a morphism in $\sC$, then also $a\in \sB$.
By induction on the degree of objects in $\sC$  it is sufficient to prove that 
if the trivial idempotent $e_A$ as in the theorem is defined for the restriction of 
$Y\colon \sC\to \CDGA^{\le0}_{A}$ to an ideal $\sB\subset\sC$, then 
$e_A$ can be extended to  the restriction of $Y$ to the ideal 
$\sB\cup\{a\}$, where $a\in \sC-\sB$ is any element of minimum degree.

The trivial idempotent of
$Y_{|\sB}$ induces a trivial idempotent on the latching functor $\sL_aY$ and then,  
according to Lemma~\ref{lem.latchingtrivialidempotent}, we have a trivial idempotent in $\CDGA^{\le0}_{A}$:
\[ L_ae\colon  X_a\otimes_{L_aX}L_aY\to X_a\otimes_{L_aX}L_aY\,.\]
Since the reduction of $L_ae$ along $B$ extends to the trivial idempotent $\epsilon_a$ of $Y_a\otimes_A B$, according to 
\cite[Theorem 6.12]{MM} there exists a trivial idempotent $e_a\colon Y_a\to Y_a$ lifting $\epsilon_a$ end extending 
$L_ae$. 
\end{proof}

As in \cite[Section 6]{MM}, Theorem~\ref{thm.lifttrivialidemp}  has a number of important consequences on the lifting of factorisations  and the push-out of deformations along trivial cofibrations. We write here only the statements, since the proofs are exactly the same, mutatis mutandis, of the corresponding results of the above mentioned paper.
For simplicity, we shall call (C,FW)-factorisation and (CW,F)-factorisation the two functorial factorisations given by model category axioms.

\begin{corollary}[cf. {\cite[Thm. 6.13]{MM}\,}]\label{cor.liftingfactorization} 
Let $A\to B$ be a surjection in $\DGArt_{\K}^{\le 0}$ and consider a morphism $f\colon P\to M$ in $\Fun(\sC,\CDGA_A^{\le 0})$ between flat diagrams.
Then every (C,FW)-factorisation of the reduction
\[ \bar{f}\colon \bar{P}=P\otimes_AB\to \bar{M}=M\otimes_AB \]
lifts to a (C,FW)-factorisation of $f$. In other words, for every factorisation $\bar{P}\xrightarrow{\sC} \bar{Q}\xrightarrow{\sF\sW} \bar{M}$ of $\bar{f}$ there exists a commutative diagram
\[ \xymatrix{  	P \ar@{->}[d]\ar@{->}[r]^{\sC} & Q \ar@{->}[d]\ar@{->}[r]^{\sF\sW} & M \ar@{->}[d] \\ 
\bar{P} \ar@{->}[r]^{\sC} & \bar{Q}\ar@{->}[r]^{\sF\sW} & \bar{M} 	} \]
in $\Fun(\sC,\CDGA_A^{\le 0})$, where the upper row reduces to the bottom row applying the functor $-\otimes_AB$ and the vertical morphisms are the natural projections.
\end{corollary}

\begin{corollary}[cf. {\cite[Thm. 6.15]{MM}\,}]\label{cor.liftingfactorization2}
Let $A\to B$ be a surjection in $\DGArt_{\K}^{\le 0}$ and consider a morphism $f\colon P\to M$ in $\Fun(\sC,\CDGA_A^{\le 0})$ between flat diagrams.
Then every (CW,F)-factorisation of the reduction
\[ \bar{f}\colon \bar{P}=P\otimes_AB\to \bar{M}=M\otimes_AB \]
lifts to a (CW,F)-factorisation of $f$. In other words, for every factorisation $\bar{P}\xrightarrow{CW} \bar{Q}\xrightarrow{F} \bar{M}$ of $\bar{f}$ there exists a commutative diagram
\[ \xymatrix{  	P \ar@{->}[d]\ar@{->}[r]^{CW} & Q \ar@{->}[d]\ar@{->}[r]^{F} & M \ar@{->}[d] \\ 
\bar{P} \ar@{->}[r]^{CW} & \bar{Q}\ar@{->}[r]^{F} & \bar{M} 	} \]
in $\Fun(\sC,\CDGA_A^{\le 0})$, where the upper row reduces to the bottom row applying the functor $-\otimes_AB$ and the vertical morphisms are the natural projections.
\end{corollary}

\begin{corollary}[cf. {\cite[Cor. 6.14]{MM}\,}]\label{cor.flattrivialcofib}
Let $A\in \DGArt_{\K}^{\le 0}$ and consider a morphism $f\colon P\to M$ in $\Fun(\sC,\CDGA_A^{\le 0})$ between flat diagrams.
Then $f$ is a Reedy cofibration if and only if its reduction $\bar{f}\colon P\otimes_A\K \to M\otimes_A \K$ is a Reedy cofibration  in $\Fun(\sC,\CDGA_\K^{\le 0})$.
\end{corollary}

\begin{corollary}[cf. {\cite[Cor. 6.16]{MM}\,}]\label{corollary.CW}
Let $A\in \DGArt_{\K}^{\le 0}$ and consider a flat diagram $P\in\Fun(\sC,\CDGA_A^{\le 0})$.
For every trivial cofibration $\bar{f}\colon \bar{P}=P\otimes_A\K\to \bar{Q}$ in $\CDGA_{\K}^{\le 0}$ there exist a flat diagram $Q\in\Fun(\sC,\CDGA_A^{\le 0})$ such that $Q\otimes_A\K=\bar{Q}$ and a trivial cofibration $f\colon P\to Q$ lifting $\bar{f}$.
\end{corollary}

\begin{corollary}[cf. {\cite[Cor. 6.17]{MM}\,}]\label{corollary.CWback}
Let $A\in \DGArt_{\K}^{\le 0}$ and consider a Reedy cofibrant diagram  $Q\in\Fun(\sC,\CDGA_A^{\le 0})$.
For every trivial cofibration $\bar{f}\colon \bar{P}\to \bar{Q}=Q\otimes_A\K$ in 
$\Fun(\sC,\CDGA_A^{\le 0})$ there exist a flat diagram $P\in\Fun(\sC,\CDGA_A^{\le 0})$ such that $P\otimes_A\K=\bar{P}$ and a lifting of $\bar{f}$ to a trivial cofibration $f\colon P\to Q$.
\end{corollary}

\bigskip
\section{The DG-Lie algebra controlling deformations of diagrams of DG-algebras over direct Reedy categories.} 
\label{sec.defodiagramalgebra}

Let $\sC$ be a fixed direct Reedy category. We have already pointed out at the end of Section~\ref{sec.diagramoverReedy} that the general  deformation theory of morphisms in strong left proper model categories applies 
to any diagram  $X\in \Fun(\sC,\CDGA^{\le0}_{\K})$ and to every diagonal thickening of Artin type
$\Delta A\to \Delta \K$, $A\in \DGArt_{\K}^{\le 0}$. In this case, for every $A\in \DGArt_{\K}^{\le 0}$  the \emph{trivial} deformation is defined as a deformation equivalent to the push-out along 
$\Delta\K\to \Delta A$: in other words,  the trivial deformation of $X$ along $A$ is represented by the diagram of differential graded algebras $a\mapsto X_a\otimes_{\K}A$.

For simplicity of notation we shall talk of deformations of $X$ over $A$ intending deformations over $\Delta A\to \Delta \K$.

It is useful to introduce the notion of strict deformation: a strict deformation of a diagram 
$X\in \Fun(\sC,\CDGA^{\le0}_{\K})$ over $A\in \DGArt_{\K}^{\le 0}$ is the data of a flat diagram 
$X_A\in \Fun(\sC,\CDGA^{\le0}_{A})$ together an isomorphism of diagrams $X_A\otimes_A\K\to X$. The 
functor $D_X$ of strict deformations of $X$ is defined by  
\[ \D_X(A)=\frac{\text{strict deformations of $X$ over $A$}}{\text{isomorphisms}}\]
and it is immediate from definitions and Nakayama's lemma that 
whenever $X$ and $A$ are concentrated in degree 0, then $D_X(A)$ are precisely the ``classical'' deformations of 
$X$ over $A$. Notice that the full subcategory of objects in $\DGArt_{\K}^{\le 0}$ concentrated in degree $0$ is exactly the usual category $\Art_{\K}$ of local Artin $\K$-algebras with residue field $\K$.

\begin{example}[Deformations of idempotent morphisms] The following trick transform the problem of deformation of 
an object equipped with an idempotent endomorphism, into the deformation problem of a diagram over a direct Reedy category. 

Denote by $\sC$ the full subcategory of $\overrightarrow{\Delta}$ having as objects the 3 finite ordinals $[0],[1], [2]$. We can visualise $\sC$ as a quiver with relations:
\begin{equation} \xymatrix{[0]\ar@(ur,ul)[rr]^{\delta_0}\ar@(dr,dl)[rr]_{\delta_1}&&[1]
\ar@(ur,ul)[rr]^{\delta_0}\ar@(dr,dl)[rr]_{\delta_2}\ar[rr]^{\delta_1}&&[2]},
\qquad \begin{matrix}\;\;\delta_0^2=\delta_1\delta_0,\\
\delta_0\delta_1=\delta_2\delta_0,\\
\;\;\delta_1^2=\delta_2\delta_1.\end{matrix}
\end{equation}
It is immediate to see that  
for any category $\bM$,  every diagram $F\colon \sC\to \bM$ such that $F(\delta_i)$ is an isomorphism
whenever $i>0$, is isomorphic to a diagram of the form:
\begin{equation}\label{equ.diagramamidempotent} 
\xymatrix{R\ar@(ur,ul)[rr]^{e}\ar@(dr,dl)[rr]_{\id}&&R
\ar@(ur,ul)[rr]^{e}\ar@(dr,dl)[rr]_{\id}\ar[rr]^{\id}&&R},
\qquad R\in \bM,\quad e^2=e\,.\end{equation}
If $\bM$ is the category of (non-graded) commutative $\K$-algebras, since isomorphisms are preserved under strict deformations, there exists a natural bijection between strict deformations of the diagram \eqref{equ.diagramamidempotent} and deformations of the pair $(R,e)$.
\end{example}

As pointed out in \cite{MM}, the functor of strict deformations is not homotopy invariant and then it is not the right object to consider: however it is very useful in order to relate the functor $\Def_X$ with classical deformations and with solutions of Maurer-Cartan equations.

\begin{lemma}\label{lem.D=Def_classic} 
In the above set-up, if $X$ is a diagram of algebras concentrated in degree 0 and 
$A\in \Art_{\K}$, then the natural map $\D_X(A)\to \Def_X(A)$ is bijective.
\end{lemma}

\begin{proof} By Nakayama's lemma, if $X_A,Y_A$ are two strict deformations of $X$, then 
$\sX,\sY$ are diagram of $A$-algebras concentrated in degree 0. In particular $\sX,\sY$ are weak equivalent in $\Fun(\sC,\CDGA^{\le0}_{A})$ is and only if they are isomorphic; this implies that 
$\D_X(A)\to \Def_X(A)$ is injective.

If $\sX \in \Fun(\sC,\CDGA^{\le0}_{A})$ is a deformation of $X$, then by the standard criterion of 
flatness in terms of relations \cite{Art,Sernesi} we have 
that for every $a\in \sC$ the $A$-algebra $H^0(\sX_a)$ is flat and the projection 
$\sX_a\to H^0(\sX_a)$ is a quasi-isomorphism. Therefore 
$H^0(\sX)$ belongs to $\D_X(A)$ and it is equivalent to $\sX$; 
this implies that 
$\D_X(A)\to \Def_X(A)$ is surjective.  
\end{proof}

\begin{lemma}\label{lem.strict=large_cofib} 
In the above set-up, if $X\in \Fun(\sC,\CDGA^{\le0}_{\K})$ is Reedy cofibrant and 
$A\in \DGArt_{\K}^{\le 0}$, then the natural map $\D_X(A)\to \Def_X(A)$ is bijective.
\end{lemma}

\begin{proof} 
We first note that if  $A\to \sX\xrightarrow{\phi} X$ is a strict deformation of $X$, then $\K\to \sX\otimes_A\K$ is a Reedy cofibration and  then, by Corollary~\ref{cor.flattrivialcofib} also $A\to \sX$ is a 
Reedy cofibration.

\emph{Injectivity.} Consider two strict deformations $A\to X_A\xrightarrow{\phi} X$ and 
$A\to Y_A\xrightarrow{\psi} X$ that are mapped in the same 
element of $\Def_X$. Notice that $\phi,\psi$ are objectwise surjective and hence fibrations.
By Proposition~\ref{prop.pushoutcfdefo} there exists a deformation $A\to Z_A\to X$ in $\Def_X(A)$ together with a commutative diagram
\[ \xymatrix{	 & A \ar@{->}[d] \ar@/_1pc/@{->}[dl] \ar@/^1pc/@{->}[dr] & \\
X_A \ar@{->}[r]^{\iota} \ar@/_1pc/@{->}[dr]_{\phi} & Z_A \ar@{->}[d]^{\eta} & Y_A \ar@{->}[l]_{\sigma} \ar@/^1pc/@{->}[dl]^{\psi} \\
 & X & 	} \]
such that $\sigma,\iota$ are Reedy trivial cofibrations.
In order to prove that  $A\to X_A\to X$ is isomorphic to $A\to Y_A\to X$,  we use the fact that, since $\sigma$ is a trivial cofibration,  the diagram of solid arrows
\[ \xymatrix{	Y_A \ar@{->}[r]^{id} \ar@{->}[d]_-{\sigma} & Y_A \ar@{->}[d]^{\psi} \\
Z_A \ar@{->}[r]^{\eta} \ar@{-->}[ur]^{\pi} & X	} \]
admits a  lifting $\pi\colon Z_A\to Y_A$. Therefore, the diagram
\[ \xymatrix{	 & A \ar@{->}[dl] \ar@{->}[dr] & \\
X_A \ar@{->}[rr]^{\pi\circ\iota} \ar@{->}[dr]_{\phi} & & Y_A \ar@{->}[dl]^{\psi} \\
 & X & 	} \]
commutes, and the reduction $\overline{\pi\iota}\colon X_A\otimes_A\K \to Y_A\otimes_A\K$ is an isomorphism. To conclude observe that $A\to \K$ is a thickening and then 
$\pi\circ\iota$ is an isomorphism too.\\

\emph{Surjectivity.} By Proposition~\ref{prop.pushoutcfdefo} it is sufficient to prove that every deformation
\[ A\xrightarrow{i} X_A\xrightarrow{\pi} X, \]
with $i$ a Reedy cofibration 
is equivalent to a strict deformation. Thus 
$X_A\otimes_A\K\xrightarrow{\pi} X$ is a weak equivalence of Reedy cofibrant diagrams and then, by the standard argument used in Ken Brown's lemma there exists a commutative diagram 
\[ \xymatrix{X_A\otimes_A\K\ar[r]^-{f}\ar[rd]_\pi&Y\ar[d]^\sigma&X\ar[l]_-{g}\ar[ld]^{\id}\\
&X&}\] 
with $f,g$ trivial cofibrations and $\sigma$ trivial fibration.
By Corollary~\ref{corollary.CW} there exists a trivial cofibration $X_A\to Y_A$ lifting 
$X_A\otimes_A\K \xrightarrow{f} Y$. By Corollary~\ref{corollary.CWback} there exists 
a flat diagram $Z_A\in \Fun(\sC,\CDGA^{\le0}_{A})$ and
a trivial cofibration $Z_A\to Y_A$ lifting 
$X \xrightarrow{g} Y$. Therefore $A\to Z_A\to X$ is a strict deformation equivalent to 
$A\xrightarrow{i} X_A\xrightarrow{\pi} X$.
\end{proof}

\begin{lemma}\label{lem.strictalgebraicallytrivial} 
In the above situation, let $A\in \DGArt_{\K}^{\le 0}$ and
$\Delta A\xrightarrow{i}\sX$ a Reedy cofibration in $\Fun(\sC,\CDGA^{\le0}_{\K})$. Denoting by $X=\sX\otimes_A\K=\sX\otimes_{\Delta A}\Delta\K$, we have a commutative diagram of diagrams
\[ \xymatrix{\Delta A\ar[d]_j\ar[r]^i&\sX\ar[d]\\
X\otimes_{\Delta\K}\Delta A\ar[ur]^-{g}\ar[r]&X}\]
where $j$ is the natural push-out map and $g$ is an isomorphism of diagrams of graded algebras. 
\end{lemma}

\begin{proof}
Consider the polynomial algebra $A[d^{-1}]\in\CDGA_{A}^{\le 0}$, where $d^{-1}$ is a variable of degree $-1$ whose differential is $d(d^{-1})=1$.  Then the natural inclusion $\alpha\colon A\to A[d^{-1}]$ is a morphism of DG-algebras, while  the natural projection $\beta\colon A[d^{-1}]\to A$ is a morphism of graded algebras; moreover $\beta\alpha$ is the identity on $A$. Since 
$\sX\to X$ is pointwise surjective, the induced morphism $\sX\otimes_A A[d^{-1}]\to 
X\otimes_A A[d^{-1}]$ is a trivial fibration,  we have a commutative diagram 
\[ \xymatrix{\Delta A\ar[d]_j\ar[r]^i&\sX\otimes_A A[d^{-1}]\ar[d]\\
X\otimes_{\Delta\K}\Delta A\ar[ur]^-{\tilde{g}}\ar[r]&X\otimes_A A[d^{-1}]}\]
and we can take $g$ as the composition of $\tilde{g}$ and $\Id\otimes\beta$.
In order to prove that $g$ is an isomorphism we can forget the differential everywhere and observe that 
the projection $A\to \K$  remains a thickening. 
\end{proof}

We can rephrase  Lemma~\ref{lem.strictalgebraicallytrivial} by saying that every strict deformation over $A$ of a cofibrant diagram $X$ is obtained by perturbing the differential of the trivial deformation $X\otimes_{\K}A$. Conversely every diagram $\sX$ of $A$-algebras obtained perturbing the differential of    
$X\otimes_{\K}A$ is pointwise flat by \cite[Prop. 7.6]{MM} and then $\sX$ is a strict deformation of $X$; (notice that this last point is false if the algebras are not concentrated in 
non-positive degrees, see \cite[Rem. 7.9]{MM}).\\

Recall that for every $R\in \CDGA^{\le0}_{\K}$, the DG-Lie algebra of derivations of $R$ is 
denoted by 
$\Der^*_{\K}(R,R)=\oplus_{i\in \Z}\Der^i_{\K}(R,R)$, where 
\[ \Der^i_{\K}(R,R)=\{\alpha\in \Hom_{\K}^i(R,R)\mid \alpha(xy)=\alpha(x)y+(-1)^{i\,\bar{x}}x\alpha(y)\}\,,\]
the bracket is the graded commutator and the differential is the adjoint operator of the differential of $R$.

We can extend naturally the above notion to every diagram $R\in \Fun(\sC,\CDGA^{\le0}_{\K})$; for every 
morphism $f\colon a\to b$ in $\sC$ we shall denote by $R_f\colon R_a\to R_b$ the induced morphism of differential graded algebras. 
Then we define 
\begin{equation}\label{equ.derivazionidiagramma}
\Der^*_{\K}(R,R)\subset \prod_{a\in \sC}\Der^*_{\K}(R_a,R_a)
\end{equation}
as the DG-Lie subalgebra of sequences $\{\alpha_a\}_{a\in \sC}$ such that for every morphism 
$f\colon a\to b$ we have $R_f \alpha_b=\alpha_a R_f$. Equivalently, an element of $\Der^*_{\K}(R,R)$ is a morphism of diagrams that is pointwise a derivation.

Any DG-Lie algebra $L$ over the field $\K$ induces
a functor 
\[ \Def_L\colon \DGArt_{\K}^{\le 0}\to \Set\]
defined in the usual way as the quotient of Maurer-Cartan element modulus gauge action \cite{H04,Man2}:
\[ \Def_{L}(A) =\dfrac{\MC_L(A)}{\text{gauge}}= \dfrac{\;\{ \eta\in (L\otimes_{\K}\mathfrak{m}_A)^1 \mid  d\eta + \frac{1}{2}[\eta,\eta] = 0 \}\;}{ \exp\,(L\otimes_{\K}\mathfrak{m}_A)^0}\,. \]

Therefore every diagram $R\in\Fun(\sC,\CDGA^{\le0}_{\K})$  
induces a deformation functor 
\[\Def_{\Der^{\ast}_{\K}(R,R)}\colon \DGArt_{\K}^{\le}\to \Set\,.\] 
In the following result we denote by $\MC_{\Der^{\ast}_{\K}(R,R)}(A)$ the set of \textbf{Maurer-Cartan elements}, i.e.,
\[ \MC_{\Der^{\ast}_{\K}(R,R)}(A) = \left\{ \eta\in (\Der^{\ast}_{\K}(R,R)\otimes_{\K}\mathfrak{m}_A)^1 \, \vert \,  d\eta + \frac{1}{2}[\eta,\eta] = 0 \right\} \, . \]

\begin{theorem}\label{thm.DefCofibrantDiagram}
Let $\sC$ be a direct Reedy category, let $R\in\Fun(\sC,\CDGA^{\le0}_{\K})$ be a Reedy cofibrant diagram and 
denote by $\D_R$ the functor  of strict deformations of $R$. 
Then for every $A\in \DGArt_{\K}^{\le 0}$ there exists a natural bijection
\[ \Def_{\Der^{\ast}_{\K}(R,R)}(A) \to \D_R(A), \]
induced by the map  
\[ \MC_{\Der^{\ast}_{\K}(R,R)}(A)\to \D_R(A),\qquad \xi\mapsto 
(R\otimes_{\K}A, d_R+\xi)\,.\]
\end{theorem}

\begin{proof} We first notice that, according to \cite[Prop. 7.7]{MM}  every diagram of type $(R\otimes_{\K}A, d_R+\xi)$, with 
$\xi \in \MC_{\Der^{\ast}_{\K}(R,R)}(A)$ is flat over $A$ and then it is a strict deformation 
of $R$, while by Lemma~\ref{lem.strictalgebraicallytrivial} every strict deformation is of this type.

The conclusion follows by observing that, the gauge equivalence corresponds to isomorphisms of diagrams of algebras whose reduction to the residue field is the identity.  In fact, given such an isomorphism $\varphi_A\colon R_A\to R_A'$ we can write $\varphi_A = \id + \eta_A$ for some $\eta_A\in(\Hom^*_{\K}(R,R)\otimes_{\K}\mathfrak{m}_A)^0$. Now, since $\K$ has characteristic $0$, we can take the logarithm to obtain $\varphi_A = e^{\theta_A}$ for some $\theta_A\in(\Der^*_{\K}(R,R)\otimes_{\K}\mathfrak{m}_A)^0$.
\end{proof}

\begin{corollary}\label{corollary.deformations}
Let $\sC$ be a direct Reedy category,  $S\in\Fun(\sC,\CDGA^{\le0}_{\K})$ a diagram and 
$R\to S$ a Reedy cofibrant replacement. Then the 
DG-Lie algebra $\Der^*_{\K}(R,R)$ controls the deformation 
functor $\Def_S$. 
\end{corollary}

\begin{proof}
By Theorem~\ref{thm.homotopyequivalence} and Lemma~\ref{lem.strict=large_cofib}  we have $\Def_S\simeq \Def_R\simeq \D_R$. The conclusion follows immediately from  Theorem~\ref{thm.DefCofibrantDiagram}.
\end{proof}

\begin{example}[Deformations of algebra morphisms]\label{ex.defoalgebramorphism}
The first application of the above results concerns deformations of a morphism of DG-algebras $f\colon B\to C$. We choose  a cofibrant resolution $p\colon R\to B$, followed by a factorisation of $fp$ as a cofibration $i\colon R\to S$ and a trivial fibration $q\colon S\to C$.
Then $i\colon R\to S$ is a Reedy cofibrant resolution of the diagram $f\colon B\to C$ and the DG-Lie algebra controlling deformations of $f$ is given by 
\[ L=\{ \alpha\in \Der_{\K}^*(S,S)\mid \alpha i(R)\subset i(R)\}\,.\] 
If one is interested to deformations where both $B,C$ remain fixed, i.e., to morphisms of DG-algebras 
of type $B\otimes A\to C\otimes A$ reducing to $f$ modulus $\mathfrak{m}_A$ we need to consider the homotopy fibre of the inclusion of DG-Lie algebras $L\to \Der_{\K}^*(R,R)\times \Der_{\K}^*(S,S)$, cf. \cite{IaconoIMNR,ManettiSemireg}.
\end{example}

\bigskip
\section{The Reedy-Palamodov resolvent and deformations of schemes}
\label{section.resolvent}

Let $X$ be a separated scheme over a field $\K$ of characteristic 0  and let $\{U_i\}_{i\in I}$ an affine open cover of $X$; actually the separatedness assumption  is only needed to ensure that every finite intersection
of elements of $\{U_i\}_{i\in I}$ is an affine open subset, therefore this hypothesis could be relaxed 
by requiring that $X$ is semi-separated and that $\{U_i\}_{i\in I}$ is a semi-separating cover.

The nerve $\sN$ of the covering is a Reedy poset with the cardinality as degree function. Denoting as usual 
by $U_{\{i_1,\ldots,i_k\}}=U_{i_1}\cap\cdots\cap U_{i_k}$, since every $U_{\alpha}$, $\alpha\in \sN$, is an affine open subset, 
the geometry of $X$ is completely determined by the diagram of $\K$-algebras 
\begin{equation}\label{equ.wtildeX} 
S_{\pal}\colon \sN\to \CDGA_{\K}^{\le 0},\qquad \alpha\mapsto 
S_{\alpha}=\Gamma(U_{\alpha},\Oh_X)\,.
\end{equation}

Since the trivial algebra $0$ is the final object in the category $\CDGA_{\K}^{\le 0}$ 
the restriction of  $S_{\pal}$ to the nerve is useful but no strictly necessary: the same works if $S_{\pal}$ is defined on the entire family of finite subsets of $I$, with
$S_{\alpha}=0$ whenever $U_{\alpha}=\emptyset$.

A deformation of $X$ over a local Artin ring $A\in \Art_{\K}$ can be interpreted as the data of a  deformation over $A$ of every open subset $U_i$ together with a deformation of the corresponding descent data. In other words, there exists a natural bijection between isomorphism classes of deformations of the scheme $X$ and isomorphism classes of deformations of the diagram 
\[U\colon \sN^{op}\to \text{affine schemes},\qquad \alpha\mapsto U_{\alpha}\,.\]
Equivalently there exists a natural bijection between isomorphism classes of deformations of $X$ and isomorphism classes of strict deformations of the diagram $S_{\pal}$. 

\begin{definition}\label{def.resolvent}
A Reedy-Palamodov resolvent of $X$, relative to an  affine open cover with nerve $\sN$ is a Reedy cofibrant resolution of the diagram $S_{\pal}$ of \eqref{equ.wtildeX}.
\end{definition}

In particular, the results of previous sections apply to this situation and 
then $\Der_{\K}^*(R,R)$ is the DG-Lie algebra controlling deformations of $X$, 
where 
$R\colon \sN\to \CDGA_{\K}^{\le 0}$ is a Reedy-Palamodov resolvent.

The name Reedy-Palamodov resolvent is clearly motivated by the large amount of common features with the usual resolvent considered in deformation theory of complex analytic spaces. 
In fact, a Reedy cofibrant resolution of $X$ over the nerve $\sN$ is a morphism of diagrams 
$R\to S_{\pal}$  over $\sN$  characterised by the following (redundant) list of properties:
\begin{enumerate}

\item for every $\alpha\in \sN$ we have 
$H^j(R_{\alpha})=0$ for every $j\not=0$ and $H^0(R_{\alpha})\cong \Gamma(U_{\alpha},\Oh_X)$;

\item for every $\alpha\in \sN$ the DG-algebra $R_\alpha\in \CDGA^{\le 0}_{\K}$ is cofibrant, 
and the natural map 
\[ \colim_{\gamma<\alpha}R_{\gamma}\to R_{\alpha}\]
is a cofibration. 

\end{enumerate}  

Replacing in the above characterization cofibrations with semifree extensions and cofibrant algebra with semifree algebra, we recover precisely the algebraic analogue of Palamodov's resolvent 
\cite{Pala1, Pala2}, also called 
 free DG-algebra resolution in \cite{BF,Fl}.  Thus we have proved the following result.

\begin{theorem}\label{thm.palamodov} 
Let $X$ be a separated scheme over a field $\K$ of characteristic 0 and let 
$R\in \Fun(\sN,\CDGA^{\le 0}_{\K})$ be a Reedy-Palamodov resolvent of $X$.  
Then the DG-Lie algebra $\Der^*_{\K}(R,R)$ controls the functor of infinitesimal deformations
of $X$. 
\end{theorem}

Writing down explicitly the resolvent can be 
very hard: in the following two illustrating examples we consider the smooth and the cuspidal rational curves, respectively.

\begin{example}[Resolvent of $\Proj^1$]  
Let $x_0,x_1$ be a set of homogeneous coordinates in $\Proj^1$, then 
a Reedy-Palamodov resolvent over the nerve of the affine cover $\{x_0\not=0\}\cup 
\{x_1\not=0\}$  is given by  
\[\begin{matrix}\xymatrix{&\K[x]\ar[d]\\
\K[y]\ar[r]&\K[x,y,e]}\end{matrix}\quad\xrightarrow{\qquad}\quad
\begin{matrix} \xymatrix{&\K[x]\ar[d]\\
\K[y]\ar[r]&\dfrac{\K[x,y]}{(xy-1)}}\end{matrix}\]
where $x=x_1/x_0$, $y=x_0/x_1$, $\deg(e)=-1$, $de=xy-1$.
\end{example}

\begin{example}[Resolvent of the cuspidal cubic]
Let  $X$ be the cuspidal cubic  in $\mathbb{P}^2$ of equation $x_0^2x_1=x_2^3$, and consider the affine open cover 
\[ X=U_0\cup U_1,\qquad U_0=\{x_0\not=0\},\quad U_1=\{x_1\not=0\}\,.\]
Then $U_{01}=\{x_0x_1\not=0\}=\{x_2\not=0\}$ and 
via the isomorphism $\C\xrightarrow{\;\simeq\;} U_0$, $w\mapsto [w,1,w^3]$, we have:
\[ \begin{split}
\Gamma(U_0,\Oh_X)&=\frac{\K[x,z]}{(z-x^3)}\simeq \K[w],\qquad x=w=\frac{x_2}{x_0},\qquad 
z=w^3=\frac{x_1}{x_0},\\
\Gamma(U_1,\Oh_X)&=\frac{\K[x,z]}{(y^2-x^3)},\qquad x=\frac{x_2}{x_1},\quad 
y=\frac{x_0}{x_1},\\
\Gamma(U_{01},\Oh_X)&=\frac{\K[t,w]}{(tw-1)},
\qquad w=\frac{x_2}{x_0},\quad 
t=\frac{x_0}{x_2}\,.\end{split}\]
The  corresponding diagram over the nerve  is: 
\[ S_{\pal}:\qquad \frac{\K[x,y]}{(y^2-x^3)}\xrightarrow{\;x\mapsto t^2,\; y\mapsto t^3\;} \frac{\K[t,w]}{(tw-1)}
\xleftarrow{\;w\mapsto w\;}\K[w]\;.\]
An easy computation shows that 
a possible Reedy-Palamodov resolvent  $p\colon R\to S_{\pal}$ is:  
\[\xymatrix{\K[x,y,e_1]\ar[d]^{p_1}\ar[r]&
\K[x,y,h,t,w,e_1,e_2,e_3,e_4]\ar[d]^{p_{01}}&\K[w]\ar[l]\ar[d]^{p_0=\id}\\
\dfrac{\K[x,y]}{(y^2-x^3)}\ar[r]^{x\mapsto t^2,\; y\mapsto t^3}&\dfrac{\K[t,w]}{(tw-1)}&\K[w]\ar[l]}\]
where: $x,y,h,t,w$ have  degree 0; $e_1,e_2,e_3,e_4$ have degree $-1$;  
\[ de_1=y^2-x^3,\; de_2=hx-1,\; de_3=tx-y,\; de_4=tw-1;\]
\[ p_{01}(x)=t^2,\; p_{01}(y)=t^3,\; p_{01}(h)=w^2,\; p_{01}(t)=t,\; p_{01}(w)=w\,.\]
It is interesting to notice that the Reedy cofibrant assumption forces to see the hyperbola $U_{01}$ as a complete intersection of 3 quadrics and a cubic in $\K^5$ and not as a plane affine conic. 
\end{example}

\bigskip
\section{The tangent and cotangent complexes}
\label{section.cotangentcomplex}

It is well known that to every noetherian separated finite-dimensional scheme $X$ over $\K$  are associated the tangent and cotangent complexes. Given a Reedy-Palamodov resolvent $R$ of $X$, the tangent complex is the  class of 
$\Der^*(R,R)$ in the homotopy category of DG-Lie algebras, and then by Theorem~\ref{thm.palamodov}  it controls the deformation theory of $X$. Its cohomology $T^*(X)$ is called tangent cohomology \cite{Pala1,Pala2}; by general results about deformation theory via DG-Lie algebras, $T^1(X)$ is the space of (classical) first order deformations, while $T^2(X)$ is the space of (classical) obstructions, cf. \cite[Thm. 5.1 and Thm. 5.2]{Pala1}.

The cotangent complex $\mathbb{L}_X$ is an object in the (unbounded) derived category of quasi-coherent sheaves 
and it can be used to compute the tangent cohomology by the formula
$T^i(X)=\Ext^i_X(\mathbb{L}_X,\Oh_X)$; moreover $\mathbb{L}_X$ has coherent cohomology and therefore each $T^i(X)$ is finite dimensional whenever $X$ is proper. The standard reference for the cotangent complex and for its application to deformations of schemes and diagrams is \cite{Illusie}.


%
If $X=\Spec(S)$ is an affine $\K$-scheme, then its cotangent complex is defined (up to quasi-isomorphism) as the sheaf associated to the $S$-module $\Omega_{R/\K}\otimes_RS$:
\[ [\widetilde{\Omega_{R/\K}\otimes_RS}] \in \D(\QCoh(X)) \]
where $R\to S$ is a cofibrant replacement in $\CDGA_{\K}^{\le 0}$ and $\Omega_{R/\K}$ denotes the DG-module of K\"ahler differentials over $R$.

According to \cite{BF,Fl}  it is possible to describe a representative of the cotangent complex in terms of 
a Reedy-Palamodov resolvent also in the non affine case.
Our goal is to present another construction  relying on  a certain model for a DG-enhancement for the unbounded derived category of quasi-coherent sheaves described in \cite{FM}.

Recall that for every DG-algebra $S\in\CDGA_{\K}^{\le 0}$ there exists a model structure on the category $\DGMod(S)$ of DG-modules where (\cite{GS,Hov99}):
\begin{itemize}
\item weak equivalences are quasi-isomorphisms,
\item fibrations are degreewise surjective morphisms,
\item a complex $\sF\in\DGMod(S)$ is cofibrant if and only if for every cospan $\sF\xrightarrow{f}\sG\xleftarrow{g}\sH$ with $g$ a surjective quasi-isomorphism there exists a lifting $h\colon \sF\to\sH$ such that $f=gh$,
\item every DG-module is fibrant,
\item cofibrations are degreewise split injective morphisms with cofibrant cokernel.
\end{itemize}

Now, let $X$ be a Noetherian separated finite-dimensional scheme over a field $\K$, fix an open affine covering $\{U_i\}_{i\in I}$ together with its nerve $\sN$ as defined in Section~\ref{section.resolvent}; consider the following diagram
\[ S_{\pal} \colon \sN \to \CDGA_{\K}^{\le 0} \; , \; \qquad \qquad \; S_{\alpha}=\Gamma(U_{\alpha},\Oh_X) \]
as already defined in~\eqref{equ.wtildeX}.
A $S_{\pal}$-module consists of the following data:
\begin{itemize}
\item an object $\sF_{\alpha}\in\DGMod(S_{\alpha})$ for every $\alpha\in\sN$,
\item a morphism $f_{\alpha\beta}\colon\sF_{\alpha}\otimes_{S_{\alpha}}S_{\beta}\to \sF_{\beta}$ in $\DGMod(S_{\beta})$ for every $\alpha\leq\beta$ in $\sN$,
satisfying the cocycle condition $f_{\beta\gamma}\circ(f_{\alpha\beta}\otimes_{S_{\beta}}\id_{S_{\gamma}})=f_{\alpha\gamma}$ for every $\alpha\leq\beta\leq\gamma$ in $\sN$.
\end{itemize}

Notice that in the above definition each map $f_{\alpha\beta}\colon\sF_{\alpha}\otimes_{S_{\alpha}}S_{\beta}\to \sF_{\beta}$ is equivalent to its adjoint morphism $\sF_{\alpha}\to\sF_{\beta}$ in $\DGMod(S_{\alpha})$, where the $S_{\alpha}$-module structure on $\sF_{\beta}$ is given by $S_{\alpha}\to S_{\beta}$.

A morphism $\varphi\colon\sF\to\sG$ between $S_{\pal}$-modules is the datum of a collection of morphisms $\{\varphi_{\alpha}\colon\sF_{\alpha}\to\sG_{\alpha}\}_{\alpha\in\sN}$ such that the diagram
\[ \xymatrix{ \sF_{\alpha}\otimes_{S_{\alpha}}S_{\beta}\ar@{->}[r]^{\varphi_{\alpha}}\ar@{->}[d]_{f_{\alpha\beta}} & \sG_{\alpha}\otimes_{S_{\alpha}}S_{\beta} \ar@{->}[d]^{g_{\alpha\beta}} \\
\sF_{\beta} \ar@{->}[r]_{\varphi_{\beta}} & \sG_{\beta} 	} \]
commutes in $\DGMod(S_{\beta})$. We shall denote by $\Hom_{S_{\pal}}(\sF,\sG)$ the set of such morphisms, and by $\Mod(S_{\pal})$ the category of $S_{\pal}$-modules.

The objects we are mainly interested in are quasi-coherent $S_{\pal}$-modules.
\begin{definition}[{\cite[Definition 3.12]{FM}}]\label{def.qcohS}
In the above notation, an $S_{\pal}$-module $\sF\in\Mod(S_{\pal})$ is called \textbf{quasi-coherent} if for every $\alpha\leq\beta$ in $\sN$ the map
\[ f_{\alpha\beta}\colon\sF_{\alpha}\otimes_{S_{\alpha}}S_{\beta}\to \sF_{\beta} \]
is a quasi-isomorphism of DG-modules over $S_{\beta}$.
\end{definition}
We shall denote by $\QCoh(S_{\pal})$ the full subcategory of quasi-coherent $S_{\pal}$-modules. Notice that the subcategory of quasi-coherent $S_{\pal}$-modules is closed both under $(C,FW)$-factorisations and $(CW,F)$-factorisations, so that it is well-defined the homotopy category $\Ho(\QCoh(S_{\pal}))$ as the Verdier quotient of cofibrant quasi-coherent $S_{\pal}$-modules modulo the class of quasi-isomorphisms. Moreover, there is a natural inclusion functor $\Ho(\QCoh(S_{\pal}))\to \Ho(\Mod(S_{\pal}))$. Definition~\ref{def.qcohS} is motivated by the following result, which was proven in~\cite[Thm. 3.9 and Thm. 5.7]{FM}.

\begin{theorem}\label{thm.equivalence}
The category $\Mod(S_{\pal})$ admits a model structure where both fibrations and weak equivalences are detected levelwise. Moreover, there exists an equivalence of triangulated categories
\[ \Upsilon^{\ast} \colon \D(\QCoh(X)) \to \Ho(\QCoh(S_{\pal})) \; , \qquad \qquad \; \Upsilon^{\ast}[\sF] = [ \{\Gamma(U_{\alpha},\sF)\}_{\alpha\in\sN}\,, ] \]
where $\Ho(\QCoh(S_{\pal}))$ denotes the homotopy category of quasi-coherent $S_{\pal}$-modules.
\end{theorem}

For a detailed discussion of the quasi-inverse of the equivalence above we refer to~\cite{FM}.
Here we only point out that the equivalence $\Upsilon^{\ast}$ commutes  in the natural way with restriction to subcoverings. 
If $\widetilde{\sN}\subset \sN$ is the nerve of a subcovering and 
$\tilde{S}_{\pal}\colon \widetilde{\sN}\to \CDGA_{\K}^{\le 0}$ is the corresponding diagram, the natural 
restriction map $\QCoh(S_{\pal})\to \QCoh(\tilde{S}_{\pal})$ is a properly defined exact functor and 
by Theorem~\ref{thm.equivalence}  the induced map 
$\Ho(\QCoh(S_{\pal}))\to \Ho(\QCoh(\tilde{S}_{\pal}))$ is an equivalence of triangulated categories.

By virtue of Theorem~\ref{thm.equivalence}, it is convenient to describe the tangent and cotangent complexes in terms of $S_{\pal}$-modules. To this aim we first need to introduce the global analogue of derivations and of K\"ahler differentials.

\subsection{Global derivations and global K\"ahler differentials}
This subsection is devoted to introduce the global versions of derivations and K\"ahler differentials, in order to define the (homotopy classes of) tangent and cotangent complexes in terms of $S_{\pal}$-modules via the equivalence of Theorem~\ref{thm.equivalence}.

We begin by defining for every morphism $\eta\colon R\to P$ of diagrams in $\Fun(\sC,\CDGA^{\le0}_{\K})$
\[ \Der^*_{\K}(R,P)\subset \prod_{a\in \sC}\Der^*_{\K}(R_a,P_a) \]
the subset of sequences $\{\alpha_a\}_{a\in \sC}$ such that for every morphism 
$f\colon a\to b$ we have $P_f \alpha_b=\alpha_a R_f$, and the structure of $R_a$-module on $P_a$ is induced by $\eta$. Notice that this is consistent with~\eqref{equ.derivazionidiagramma}. Similarly one can define 
\[ \Hom^{\ast}_{S_{\pal}}(\sF,\sG)\subset \prod_{a\in \sN}\Hom^{\ast}_{S_{\pal}}(\sF_{\alpha},\sG_{\alpha}) \]
for every $\sF,\sG\in\Mod(S_{\pal})$.

It is clear that every diagram $R\xrightarrow{\eta} P\xrightarrow{\mu}T$ induces by composition two morphisms 
\[ \Der^*_{\K}(R,P)\xrightarrow{\mu_*}  \Der^*_{\K}(R,T)\xleftarrow{\eta^*}\Der^*_{\K}(P,T)\]
which formally satisfy the usual properties of derivations in the model category $\CDGA^{\le0}_{\K}$. 

Recall that by~\cite{Hin,QuillenCR} for any given $S\in\CDGA_{\K}^{\leq0}$, the functor of K\"ahler differentials admits a Quillen right adjoint given by the trivial extension:
\[ \Omega_{-}\otimes_{-}S\colon \CDGA_{\K}^{\le 0}\downarrow S \rightleftarrows \DGMod^{\le 0}(S) \colon -\oplus S \; . \]
This adjoint pair easily generalizes to the case of diagrams. Let $X$ be a Noetherian separated finite-dimensional scheme over a field $\K$, fix an open affine covering $\{U_i\}_{i\in I}$ together with its nerve $\sN$ as defined in Section~\ref{section.resolvent}. Now consider the corresponding diagram $S_{\pal}\in\Fun(\sN, \CDGA_{\K}^{\le 0})$ defined by
\[ S_{\pal} \colon \sN \to \CDGA_{\K}^{\le 0} \; , \; \qquad \qquad \; S_{\alpha}=\Gamma(U_{\alpha},\Oh_X) \]
as in~\eqref{equ.wtildeX}. 
Then define the functor $\Omega^{\sN}_{-}\otimes_{-}S_{\pal}\colon \Fun(\sN,\CDGA_{\K}^{\le 0})\downarrow S_{\pal} \to \Mod^{\le 0}(S_{\pal})$ as follows:
\begin{itemize}
\item $\Mod^{\le 0}(S_{\pal})\subseteq\Mod(S_{\pal})$ denotes the full subcategory of $S_{\pal}$-modules concentrated in non-negative degrees; it admits a model structure such that any cofibrant object $X\in\Mod^{\le 0}(S_{\pal})$ is also cofibrant when regarded as an object in $\Mod(S_{\pal})$,~\cite[Rem. 3.11]{FM}.
\item $\left(\Omega^{\sN}_{R}\otimes_{R}S_{\pal}\right)_{\alpha} = \Omega_{R_{\alpha}}\otimes_{R_{\alpha}}S_{\alpha}$ for every $R\in\Fun(\sN,\CDGA_{\K}^{\le 0})$ and every $\alpha\in\sN$.
\item for every $\alpha\leq \beta$ the map
\[ \Omega_{R_{\alpha}}\otimes_{R_{\alpha}}S_{\alpha}\otimes_{S_{\alpha}}S_{\beta} = \Omega_{R_{\alpha}}\otimes_{R_{\alpha}}S_{\beta}\to \Omega_{R_{\beta}}\otimes_{R_{\beta}}S_{\beta} \]
is the one induced by K\"ahler differentials.
\end{itemize}

Notice that in the setting of Definition~\ref{def.cotangentcomplex} we have $\Omega_R^{\sN}\otimes_RS_{\pal}=\sL_R$. Moreover, there exists a bi-natural isomorphism
\[ \Hom_{ S_{\pal} }^{\ast}(\Omega_P^{\sN}\otimes_PS_{\pal}, \sF) \cong \Der^{\ast}_{\K}(P,\sF) \]
for every $P\in \Fun(\sN,\CDGA_{\K}^{\le 0})$ and every $\sF\in\Mod^{\le 0}(S_{\pal})$.

\begin{remark}\label{rmk.cofibrantcotangent}
It is easy to show that the functor $\Omega^{\sN}_{-}\otimes_{-}S_{\pal}\colon \Fun(\sN,\CDGA_{\K}^{\le 0})\downarrow S_{\pal} \to \Mod^{\le 0}(S_{\pal})$ defined above admits a right Quillen adjoint as in the affine case. In particular, $\Omega^{\sN}_{-}\otimes_{-}S_{\pal}$ maps Reedy-cofibrant diagrams to cofibrant $S_{\pal}$-modules. Moreover, by Ken Brown's Lemma it preserves weak equivalences between cofibrant objects.
\end{remark}

\begin{lemma}\label{lemma.derivations1}
In the above setup, if $R$ is Reedy cofibrant and $\mu$ is a weak-equivalence, then $\Der^*_{\K}(R,P)\xrightarrow{\mu_*}  \Der^*_{\K}(R,T)$ is a quasi-isomorphism.
Moreover, if $\eta\colon R\to P$ is a weak equivalence between cofibrant objects in $\Fun(\sN,\CDGA_{\K}^{\le 0})$, 
then:
\begin{enumerate}
\item the map $\eta^{\ast}\colon \Der^{\ast}_{\K}(P,T)\to\Der_{\K}^{\ast}(R,T)$ defined above is a quasi-isomorphism,
\item $\Der_{\K}^{\ast}(R,R)$ and $\Der_{\K}^{\ast}(P,P)$ are quasi-isomorphic as DG-Lie algebras.
\end{enumerate}

\end{lemma} 
\begin{proof}
It  is an easy consequence of Remark~\ref{rmk.cofibrantcotangent} and of standard arguments, see 
\cite[Rem. 6.9]{MM} and \cite[Lemma 6.14]{FM}. 
\end{proof}

In particular, by Lemma~\ref{lemma.derivations1} it follows that in the setting of Corollary~\ref{corollary.deformations} the homotopy class of the DG-Lie algebra $\Der_{\K}^{\ast}(R,R)$ does not depend on the choice of the Reedy cofibrant replacement $R$.

\begin{example}\label{example.defmorphism2}
Let $f\colon R\to S$ be a surjective morphism of cofibrant DG-algebras with ideal $I$. Then the deformations of $f$ are controlled by the DG-Lie algebra 
\[ M=\{\alpha\in  \Der_{\K}^*(R,R)\mid \alpha(I)\subset I\}\cong
\Der^*_{\K}(R,R)\times_{\Der^*_{\K}(R,S)}\Der^*_{\K}(S,S)\,.\]

Consider a factorisation $f\colon R\xrightarrow{i}H\xrightarrow{p}S$ with $i$ a cofibration and 
$p$ a trivial fibration. Then $R\to H$ is a Reedy cofibrant resolution of $R\to S$ and 
we need to prove that the DG-Lie algebra $M$ is quasi-isomorphic to 
\[ L=\Der^*_{\K}(R,R)\times_{\Der^*_{\K}(R,H)}\Der^*_{\K}(H,H)\,.\] 
The obvious composition maps give a commutative diagram of complexes
\[\xymatrix{
\Der^*_{\K}(R,R)\ar@{=}[d]\ar[r]&\Der^*_{\K}(R,H)\ar@{>>}[d]&\Der^*_{\K}(H,H)\ar@{>>}[d]\ar@{>>}[l]\\
\Der^*_{\K}(R,R)\ar@{=}[d]\ar@{>>}[r]&\Der^*_{\K}(R,S)&\Der^*_{\K}(H,S)\ar@{>>}[l]\\
\Der^*_{\K}(R,R)\ar@{>>}[r]&\Der^*_{\K}(R,S)\ar@{=}[u]&\Der^*_{\K}(S,S)\ar[u]\ar[l]}\]
where the double head arrows denote  surjective morphisms and every vertical arrow is a quasi-isomorphism. Notice that the assumption that $S$ is cofibrant is used to ensure that the map
$\Der^*_{\K}(S,S)\to \Der^*_{\K}(H,S)$ is a quasi-isomorphism.
By coglueing lemma we have two quasi-isomorphisms 
\[ \xymatrix{L\ar[r]^-{l}& \Der^*_{\K}(R,R)\times_{\Der^*_{\K}(R,S)}\Der^*_{\K}(H,S)&M\ar[l]_-{m}}\]
and an easy  direct inspection shows that $l$ is surjective. 
In order to finish the proof it is sufficient to observe that the fibre product of the above cospan
is the  DG-Lie algebra of derivations of the diagram $R\to H\to S$,  defined as in \eqref{equ.derivazionidiagramma}.
\end{example}

\begin{example}[Derived functor of points]
Given a morphism $f\colon B\to \K$ in $\CDGA^{\le 0}_{\K}$ we are interested to morphisms in the 
homotopy category 
\[ B\to A\qquad A\in \DGArt_{\K}^{\le 0}\]
lifting $f$. Given a cofibrant resolution $p\colon R\to B$, the above morphisms can be interpreted
as deformations of the diagram $fp\colon  R\to \K$ inducing a trivial deformation of $R$. 
Denoting by $I\subset R$ the kernel of $fp$, by Examples~\ref{ex.defoalgebramorphism} and \ref{example.defmorphism2} the corresponding DG-Lie algebra is
equal to the homotopy fibre of the inclusion 
\[ \{\alpha\in  \Der_{\K}^*(R,R)\mid \alpha(I)\subset I\}\to \Der_{\K}^*(R,R)\,.\]
\end{example}

\bigskip

\subsection{The quasi-coherent $S_{\pal}$-module corresponding to the cotangent complex}

Let $X$ be a Noetherian separated finite-dimensional scheme over a field $\K$, with a fixed open affine covering $\{U_i\}_{i\in I}$. Consider the diagram $S_{\pal}$ as above together with a cofibrant replacement $R\to S_{\pal}$ in $\Fun(\sN,\CDGA_{\K}^{\le 0})$. Hence, according to Definition~\ref{def.resolvent}, $R$ is a Reedy-Palamodov resolvent for $X$.

Notice that by Lemma~\ref{lemma.derivations1} the tangent complex $\Der^{\ast}_{\K}(R,R)$ 
is well-defined in the homotopy category of DG-Lie algebras, i.e., it does not depend on the Reedy-Palamodov resolvent $R$. Moreover, it is quasi-isomorphic (as a complex) to  $\Der^{\ast}_{\K}(R,S_{\pal})$.

For what concerns the cotangent complex, we shall make use of the equivalence of Theorem~\ref{thm.equivalence}, so that we introduce the definition in terms of the homotopy category of quasi-coherent $S_{\pal}$-modules.

\begin{definition}[The cotangent complex]\label{def.cotangentcomplex}
In the above notation, define the cotangent complex to be the class $[\sL_R]\in\Ho(\QCoh(S_{\pal}))$, where the $S_{\pal}$-module $\sL_R$ is defined by:
\begin{itemize}
\item $\sL_{R,\alpha} = \Omega_{R_{\alpha}}\otimes_{R_{\alpha}}S_{\alpha}$ for every $\alpha\in\sN$,
\item  for every $\alpha\leq\beta$ the map $l_{\alpha\beta}\colon \sL_{R,\alpha}\otimes_{S_{\alpha}}S_{\beta} \to \sL_{R,\beta}$ is obtained applying the functor $\Omega_{-}\otimes_{-}S_{\beta}$ to the map $R_{\alpha}\to R_{\beta}$.
\end{itemize}
\end{definition}

Observe that by Remark~\ref{rmk.cofibrantcotangent} the homotopy class $[\sL_R]$ does not depend on the choice of the resolvent. Therefore, in order to prove that Definition~\ref{def.cotangentcomplex} is well-posed we only need to show that the $S_{\pal}$-module $\sL_R$ is quasi-coherent in the sense of Definition~\ref{def.qcohS}. We proceed by proving a series of preliminary lemmas.

The assumptions in the following lemma are motivated by the fact that, if $U=\Spec(A)$ is an affine scheme and $V=\Spec(B)\subset U$ is an open affine subset, then  the morphism $A\to B$ is flat and the natural map $B\otimes_A B\to B$ is an isomorphism.

\begin{lemma}\label{lemma.cotangent1}
Let $A\to B$ be a flat morphism in $\CDGA_{\K}^{\le 0}$ such that the natural map $B\otimes_A B\to B$ is a weak equivalence.
Consider a commutative diagram 
\[ \xymatrix{R\ar[r]^i\ar[d]&S\ar[d]\\
A\ar[r]&B}\]
with the vertical arrow cofibrant replacements and $i$ a cofibration.
Then $\Omega_R\otimes_RB\to \Omega_S\otimes_S  B$ is a trivial cofibration  of $B$-modules.
\end{lemma}

\begin{proof}

Let $j\colon S\to S\otimes_RS$ be the push-out of $i$ by itself. We first show that $j$ is  a trivial cofibration. 
By model category axioms cofibrations are closed under pushouts, so that we only need to prove that $j$ is a weak equivalence. Since the category $\CDGA^{\le 0}_{\K}$ is left proper and $i\colon R\to S$ is a cofibration, the natural maps
\[ S\otimes_R S\to S\otimes_R B,\qquad S=S\otimes_R R\to S\otimes_R A,\]
are weak equivalences. By the universal property of push-out we have a diagram 
\[\xymatrix{R\ar[d]^f\ar[r]^i&S\ar[d]^g\ar@(r,u)[ddr]&\\
A\ar[r]\ar@(d,l)[rrd]&B&\\
&&S\otimes_R A\ar[ul]_-{q}}\]
and $q$ is a weak equivalence by the 2 out of 3 property.
Now, since $A\to B$ is flat, the composite map 
\[ S\otimes_R S\to S\otimes_R B=(S\otimes _R A)\otimes_A B\to B\otimes_A B\]
is a weak equivalence. Therefore the lemma follows by the 2 out of 3 property  applied to the commutative diagram
\[\xymatrix{S\ar[r]^-{j}\ar[dr]_{\Id}&S\otimes_RS\ar[d]\ar[r]&B\otimes_AB\ar[d]&\\
&S\ar[r]&B&.}\]

By Remark~\ref{rmk.cofibrantcotangent} the morphism
$\Omega_R\otimes_RB\to \Omega_S\otimes_S  B$ is a cofibration (hence injective); in view of the standard exact sequence 
\[ \Omega_R\otimes_RB\to \Omega_S\otimes_S  B\to \Omega_{S/R}\otimes_S B\to 0\,,\]
it is sufficient to show that $\Omega_{S/R}\otimes_S B$ is acyclic.
Since $S\to S\otimes_R S$ is a trivial cofibration the module 
\[ \Omega_{S\otimes_RS/S}=\Omega_{S/R}\otimes_RS=\Omega_{S/R}\otimes_SS\otimes_RS\]
is acyclic. Since $S\otimes_RS\to 
B\otimes_AB\to B$ is a weak equivalence and 
$\Omega_{S/R}$ is cofibrant as an $S$-module, there exists a weak equivalence 
\[ \Omega_{S/R}\otimes_SS\otimes_RS\to \Omega_{S/R}\otimes_S B\,.\]
\end{proof}

\begin{proposition}\label{prop.cotangentqcoh}
Let $A\to B$ be a flat morphism in $\CDGA_{\K}^{\le 0}$ such that the natural map $B\otimes_A B\to B$ is a weak equivalence.
Consider a commutative diagram 
\[ \xymatrix{R\ar[r]^i\ar[d]&S\ar[d]\\
A\ar[r]&B}\]
with the vertical arrow cofibrant replacements.
Then $\Omega_R\otimes_RB\to \Omega_S\otimes_S  B$ is a weak equivalence.
\end{proposition}

\begin{proof}  By taking a (C,FW)-factorisation of the natural map $R\otimes S\to S$ we have the diagrams 
\[ \xymatrix{&T\ar[d]^{p}\\
R\ar[r]^i\ar[ur]^j\ar[d]&S\ar[d]^g\\
A\ar[r]&B}
\qquad \qquad
\xymatrix{ S\ar[r]^h\ar[d]_{g} & T\ar[d]^{g\circ p} \\
B\ar[r]_{\id} & B 	} \]
with $j$ cofibration and $p$ a trivial fibration admitting a  cofibration $h\colon S\to T$ as a section.
Hence by Lemma~\ref{lemma.cotangent1} the maps
\[ \Omega_R\otimes_RB\to \Omega_T\otimes_TB \leftarrow \Omega_S\otimes_SB \]
induced by $j$ and $h$ are trivial cofibrations.
Moreover, the composition
\[ \Omega_S\otimes_SB \to \Omega_T\otimes_TB \to \Omega_S\otimes_SB \]
is the identity being $h$ a section of $p$; therefore $\Omega_T\otimes_TB \to \Omega_S\otimes_SB$ is a weak equivalence.
The statement follows by considering the composition $\Omega_R\otimes_RB\to \Omega_T\otimes_TB \to \Omega_S\otimes_SB$.

\end{proof}

We are now ready to prove that Definition~\ref{def.cotangentcomplex} is well-posed. Recall that by Definition~\ref{def.qcohS} a $S_{\pal}$-module $\sF\in\Mod(S_{\pal})$ is called quasi-coherent if the map $f_{\alpha\beta}\colon \sF_{\alpha}\otimes_{S_{\alpha}}S_{\beta}\to S_{\beta}$ is a quasi-isomorphism for every $\alpha\leq\beta$ in $\sN$.

\begin{theorem}
Let $X$ be a Noetherian separated finite-dimensional scheme over a field $\K$, with a fixed open affine covering $\{U_i\}_{i\in I}$. Consider the associated diagram $S_{\pal}$ as in~\eqref{equ.wtildeX}, together with a Reedy-Palamodov resolvent $R\to S_{\pal}$. Then the $S_{\pal}$-module $\sL_R$ defined in~\ref{def.cotangentcomplex} is cofibrant and quasi-coherent.
\begin{proof}
Since $\sL_R=\Omega_R^{\sN}\otimes_RS_{\pal}$, the statement immediately follows by Remark~\ref{rmk.cofibrantcotangent} and Proposition~\ref{prop.cotangentqcoh}.
\end{proof}
\end{theorem}

Because of Theorem~\ref{thm.palamodov}, it is important to concretely understand the homotopy class of the DG-Lie algebra of derivations of a Reedy-Palamodov resolvent. To this aim, the next result relates the cohomology of derivations of a Reedy-Palamodov resolvent with the cotangent complex.

\begin{theorem}\label{thm.cohomology}
Let $X$ be a Noetherian separated finite-dimensional scheme over a field $\K$, with a fixed open affine covering $\{U_i\}_{i\in I}$ and its nerve $\sN$. Consider the associated diagram $S_{\pal}$ as in~\eqref{equ.wtildeX}, together with a Reedy-Palamodov resolvent $R\to S_{\pal}$.  Then there exist an isomorphism
\[ H^{\ast}\left( \Der_{\K}^{\ast}(R,R) \right) \cong H^{\ast}\left( \Hom^{\ast}_{S_{\pal}}(\sL_R,S_{\pal})\right) \; . \]
\begin{proof}
By Lemma~\ref{lemma.derivations1} the map
\[ \Der_{\K}^{\ast}(R,R) \to \Der_{\K}^{\ast}(R,S_{\pal})\cong \Hom^{\ast}_R(\Omega_R^{\sN},S_{\pal}) \cong \Hom^{\ast}_{S_{\pal}}(\Omega_R^{\sN}\otimes_RS_{\pal},S_{\pal}) \]
is a quasi-isomorphism.
\end{proof}
\end{theorem}

It is worth to point out that combining Theorem~\ref{thm.palamodov} and Theorem~\ref{thm.cohomology} we have  that the space of first order deformations is nothing but $\Ext^1_{X}(\mathbb{L}_X,\Oh_X)$ and the obstructions are contained in $\Ext^2_{X}(\mathbb{L}_X,\Oh_X)$, where $\mathbb{L}_X\in 
\D(\QCoh(X))$ and $\Upsilon^{\ast}\mathbb{L}_X =[\sL_R]$.

\bigskip

\section{A remark about deformations of maps (after Horikawa, Ran and Pridham)}

The argument used  in Section~\ref{section.resolvent}  easily extends to every morphism of separated schemes, considered as a contravariant functor from the direct Reedy category $\{0\to 1\}$ to the category of separated schemes. More precisely,
given any morphism $f\colon X\to Y$ of separated schemes we can find a family of pairs 
$\{(U_i,V_i)\}_{i\in I}$ such that:
\begin{enumerate}

\item the family $\{U_i\}$ is an affine open cover of $X$,

\item the family $\{V_i\}$ is an affine open cover of $Y$;
    
\item $f(U_i)\subset V_i$ for every $i\in I$.
\end{enumerate}
We then define the nerve $\sN$ as the family of finite subsets $\alpha\subset I$ such that 
$V_{\alpha}\not=\emptyset$. Then the data $X,Y,f$ is encoded by the diagram  
\[ S_{\pal,\pal}\colon \sN\times \{0\to 1\}\to \CDGA^{\le 0}_{\K},\]
\[ S_{\alpha,0}=\Gamma(V_{\alpha},\Oh_Y),\qquad 
S_{\alpha,1}=\Gamma(U_{\alpha},\Oh_X)\,.\]
By definition of the Reedy model structure, a diagram 
$Q_{\alpha,j}$, $\alpha\in \sN$, $j=0,1$ over the small category 
$\sN\times \{0\to 1\}$ is cofibrant if and only if for every 
$\alpha\in \sN$ the maps  
\[ \colim_{\gamma<\alpha}Q_{\gamma,0}\to Q_{\alpha,0},\qquad 
\colim_{\gamma<\alpha}Q_{\gamma,1}\otimes_{\mathop{\colim\limits_{\gamma<\alpha}}Q_{\gamma,0}}Q_{\alpha,0}\to Q_{\alpha,1} \]
are cofibrations in $\CDGA_{\K}^{\le 0}$. Therefore $Q_{\pal,\pal}$ is cofibrant if and only if 
$Q_{\pal,0}$ is cofibrant and $Q_{\pal,0}\to Q_{\pal,1}$ is a cofibration. 

In particular if $R_{\pal,\pal}$  is a resolvent for $f\colon X\to Y$, i.e., a Reedy cofibrant resolution of $S_{\pal,\pal}$, then $R_{\pal,0}$ is a resolvent of 
$Y$, $R_{\pal,1}$ is a resolvent of $X$ and $R_{\pal,0}\to R_{\pal,1}$ is a cofibration, 
cf. Example~\ref{ex.defoalgebramorphism}. This implies for instance, considering $R_{\pal,1}$ as a
$R_{\pal,0}$-module, that the natural map
\[ \sigma\colon \Der^*(R_{\pal,1},R_{\pal,1})\to \Der^*(R_{\pal,0}, R_{\pal,1})\]
is surjective, while the natural map   
\[ \tau\colon \Der^*(R_{\pal,0},R_{\pal,0})\to \Der^*(R_{\pal,0}, R_{\pal,1})\]
is injective. Therefore we have a short exact sequence
\[ 0\to \Der^*(R_{\pal,\pal},R_{\pal,\pal})\to 
\Der^*(R_{\pal,0},R_{\pal,0})\oplus \Der^*(R_{\pal,1},R_{\pal,1})\xrightarrow{\sigma-\tau}
\Der^*(R_{\pal,0},R_{\pal,1})\to 0\,.\]
We have proved that the DG-Lie algebra on the left controls the deformations of $f$. Thus, setting 
$T^i(f)=H^i(\Der^*(R_{\pal,\pal},R_{\pal,\pal}))$ we have that $T^1(f)$ is the space of first order deformations, while $T^2(f)$ is the obstruction space. The resulting cohomology long exact sequence 
\begin{equation}\label{equ.horikawaran} 
\cdots T^i(f)\to T^i(X)\oplus T^i(Y)\to H^i(\Der^*(R_{\pal,0},S_{\pal,1}))\to T^{i+1}(f)\to 
T^{i+1}(X)\oplus T^{i+1}(Y)\cdots
\end{equation} 
is familiar to most people working in deformation theory,  since the same has been  
proved over the field of complex numbers by Horikawa \cite{Horidhm2} in the smooth case (see also 
\cite[p. 184]{namba} and \cite[Rem. 5.2]{IaconoIMNR}),
by Ran \cite{Ran} in the reduced case and in full generality by Pridham (Theorem 3.2 and Lemma 3.3 applied to Example 3.3 of \cite{pridham}).

The same considerations hold, mutatis mutandis, for every diagram 
$X\colon \sC^{op}\to \mathbf{Schemes}$ of separated $\K$-schemes over the opposite of a direct Reedy category: for simplicity of notation, for every morphism $f\colon a\to b$ in $\sC$ we shall use the same symbol $f\colon X_b\to X_a$ to denote the corresponding morphism of schemes 
$X(f)$. Here 
the role of affine open subsets is played by elements 
\[ U_{\pal}\in \prod_{a\in \sC}\{ \text{affine open subsets of }X_a\}\]
such that $f(U_b)\subset U_a$ for every morphism $f\colon a\to b$ in $\sC$.
The fact that $\sC$ is Reedy direct easily implies that 
there exists a ``covering'' $\{U_{i,\pal}\}$ of $X$ made by elements as above, with corresponding nerve $\sN$. 
Finally the deformations of the diagram of schemes $X$ are the same as the deformations of the diagram
of algebras 
\[S_{\pal,\pal}\colon \sN\times \sC\to \CDGA^{\le 0}_{\K},\qquad 
S_{\alpha,a}=\Gamma(U_{\alpha,a},\Oh_{X_a})\,,\]
that can be studied as in Section~\ref{sec.defodiagramalgebra}, since $\sN\times \sC$ is direct Reedy.

\begin{acknowledgement}
Both authors thank Donatella Iacono for useful discussions on the subject of this paper.  
\end{acknowledgement}

\end{document}